\documentclass{scrartcl}
\usepackage[numbers]{natbib}

\usepackage{hyperref}

\usepackage{amsmath}
\usepackage{amsthm}
\usepackage{amssymb}

\usepackage{tikz} 
\usepackage{tikz-cd}

\theoremstyle{plain}
\newtheorem{theorem}{Theorem}[section]
\newtheorem{lemma}[theorem]{Lemma}
\newtheorem{proposition}[theorem]{Proposition}
\newtheorem{corollary}[theorem]{Corollary}

\theoremstyle{definition}
\newtheorem{definition}[theorem]{Definition}

\newtheorem{remark}[theorem]{Remark}

\makeatletter
\def\blfootnote{\gdef\@thefnmark{}\@footnotetext}
\makeatother

\newcommand{\A}{\mathbb{A}}

\newcommand{\F}{\mathbb{F}}
\newcommand{\G}{\mathbb{G}}
\newcommand{\Q}{\mathbb{Q}}
\newcommand{\Z}{\mathbb{Z}}

\newcommand{\fraka}{\mathfrak{a}}
\newcommand{\frakm}{\mathfrak{m}}
\newcommand{\frakA}{\mathfrak{A}}

\newcommand{\frakU}{\mathfrak{U}}
\newcommand{\frakV}{\mathfrak{V}}
\newcommand{\frakW}{\mathfrak{W}}
\newcommand{\frakX}{\mathfrak{X}}
\newcommand{\frakY}{\mathfrak{Y}}
\newcommand{\frakZ}{\mathfrak{Z}}

\newcommand{\calA}{\mathcal{A}}
\newcommand{\calC}{\mathcal{C}}

\newcommand{\calE}{\mathcal{E}}
\newcommand{\calF}{\mathcal{F}}
\newcommand{\calG}{\mathcal{G}}
\newcommand{\calH}{\mathcal{H}}
\newcommand{\calI}{\mathcal{I}}
\newcommand{\calJ}{\mathcal{J}}
\newcommand{\calM}{\mathcal{M}}
\newcommand{\calO}{\mathcal{O}}

\newcommand{\calU}{\mathcal{U}}
\newcommand{\calV}{\mathcal{V}}
\newcommand{\calW}{\mathcal{W}}
\newcommand{\calX}{\mathcal{X}}
\newcommand{\calY}{\mathcal{Y}}

\newcommand{\FS}{\mathrm{FS}}
\newcommand{\Fin}{\mathrm{Fin}}
\newcommand{\Rig}{\mathrm{Rig}}

\newcommand{\rig}{\mathrm{rig}}

\newcommand{\Mod}{\mathrm{Mod}}
\newcommand{\Coh}{\mathrm{Coh}}
\newcommand{\Cat}{\mathrm{Cat}}
\newcommand{\Set}{\mathrm{Set}}

\newcommand{\Ad}{\mathrm{Ad}}
\newcommand{\ad}{\mathrm{ad}}
\newcommand{\Zar}{\mathrm{Zar}}
\newcommand{\zar}{\mathrm{zar}}
\newcommand{\Tot}{\mathrm{Tot}}
\newcommand{\tot}{\mathrm{tot}}
\newcommand{\Bl}{\mathrm{Bl}}
\newcommand{\op}{\mathrm{op}}

\newcommand{\Spec}{\mathrm{Spec}}
\newcommand{\Proj}{\mathrm{Proj}}
\newcommand{\Spf}{\mathrm{Spf}}
\newcommand{\Spwf}{\mathrm{Spwf}}

\newcommand{\Frob}{\mathrm{Frob}}
\newcommand{\GL}{\mathrm{GL}}
\newcommand{\Hom}{\mathrm{Hom}}

\newcommand{\Tr}{\mathrm{Tr}}
\newcommand{\tr}{\mathrm{tr}}
\newcommand{\N}{\mathrm{N}}
\newcommand{\Gal}{\mathrm{Gal}}

\newcommand{\et}{\mathrm{\acute{e}t}}
\newcommand{\spe}{\mathrm{sp}}
\newcommand{\id}{\mathrm{id}}
\newcommand{\wt}{\mathrm{wt}}

\title{Weak Analytic Geometry and a Trace Formula for Families of $p$-adic Representations}
\author{James Upton\\Department of Mathematics\\University of California, Irvine, CA 92697-3875\\jtupton@uci.edu}
\date{\today}

\begin{document}

	\maketitle
	
	\begin{abstract}
		The \emph{eigencurve} is a powerful tool introduced by Coleman and Mazur to study $p$-adic families of overconvergent modular forms. In this article, we introduce an analogous set of tools for understanding families of ``overconvergent'' $p$-adic representations of $\pi_1(X)$, where $X$ is a smooth affine variety over a finite field of characteristic $p$. Our main theorem is a trace formula relating the $L$-function of such a family to the geometry of a sequence of associated eigenvarieties. In the case of a single $p$-adic representation, our result reduces to the well known trace formula of Monsky. We apply our theory to the study of $T$-adic exponential sums attached to $\Z_p$-towers over $X$. Special cases of this theory have been applied by Davis, Wan, and Xiao to prove a spectral halo decomposition of the eigencurve attached to $\Z_p$-towers over $X=\A^1$.
	\end{abstract}
	
	\tableofcontents
	
\section{Introduction}
	
	Let $k$ be a finite field of characteristic $p$. In their seminal article \cite{Coleman}, Coleman and Mazur construct a rigid analytic curve parameterizing the finite-slope overconvergent $p$-adic eigenforms of tame level $1$. This \emph{eigencurve} and its generalizations are useful in particular for studying the Igusa tower over the ordinary locus of modular curves. The global geometry of the eigencurve is somewhat mysterious, however Coleman has conjectured a ``spectral halo'' property describing its behavior near the boundary of weight space, see e.g. \cite{LWX} for a precise formulation of this conjecture. In \cite{Davis}, Davis, Wan, and Xiao introduce an analogous eigencurve to study $\Z_p$-towers over the affine line $\A^1_k$, and moreover prove the analogue of the spectral halo property for their eigencurve. Their curve may be viewed a hypersurface in the rigid analytic space $\calW \times \G_m^\rig$, where $\calW$ is the \emph{weight space} of continuous $p$-adic characters of $\Z_p$. An essential component of their work is a variant of the Dwork trace formula, which expresses $p$-adic exponential sums over $k$ in terms of the $\G_m^\rig$-coordinates of the points of the eigencurve. The spectral halo property of $\calC$ then implies a strong periodicity property for the zeta functions of certain Artin-Schreier-Witt coverings of $\A^1_k$. Their methods have since been adapted to prove a spectral halo property for the eigencurve associated to automorphic forms for a definite quaternion algebra over $\Q$ \cite{LWX}.
	
	To shed some light on Coleman's conjecture, it would be interesting to know to what extent the results of Davis \emph{et al.} can be generalized. As noted in \cite{Ren}, there are essentially two paths of generalization: (1) we may replace the weight space $\calW$ by another rigid analytic space, or (2) we may replace the base variety $\A^1_k$ with another $k$-variety. Those authors consider case (1), replacing $\calW$ by the space of continuous $p$-adic characters of $\Z_p^d$ and constructing a $d$-dimensional \emph{eigenvariety} to study $\Z_p^d$-coverings of the affine line. The goal of this article is to introduce the analogous theory in case (2), replacing $\A^1_k$ by a general smooth affine variety $X/k$. We will also work relative to a fairly general weight space: throughout, we assume that $R$ is a complete Noetherian local domain with maximal ideal $\frakm$ and residue field $\F_q \subseteq k$. The \emph{weight space} $\calW$ will be the rigid analytic space associated to $R$. The study of $\Z_p^d$-towers over $X$ corresponds to the case $R=\Z_p[[T_1,...,T_d]]$, but we may apply the theory to more general $R$ to study e.g. deformations of $p$-adic Galois representations.
	
	Since $X$ is affine and smooth, it is well known that there is a unique smooth formal $R$-scheme $\frakX_\infty$ lifting $X$. Fix a geometric point $\bar{x}$ of $X$, and consider a continuous representation
	\begin{equation}\label{eq:rep}
		\rho:\pi_1(X,\bar{x})	\to	\GL_n(R).
	\end{equation}
	If $P$ is a point in our weight space, then specializing any such representation at $P$ yields an ordinary $p$-adic representation of $\pi_1(X,\bar{x})$. Thus we regard $\rho$ as a continuous family of $p$-adic representations parameterized by $\calW$. In Section \ref{s:smod}, we identify the category all such representations with a full subcategory of the category of \emph{$\sigma$-modules} over $X$. Informally, these are pairs $(M,\phi)$, where $M$ is a locally free $\calO_\frakX$-module and $\phi$ is a ``Frobenius structure'' on $M$. To each $\sigma$-module $(M,\phi)$, we attach an $L$-function $L(\phi,s) \in R[[s]]$. If the expression of $\phi$ in coordinates is overconvergent with respect to the $\frakm$-adic topology, then we say that $(M,\phi)$ is an \emph{overconvergent} $\sigma$-module. Our main theorem (Theorem \ref{t:merom} below) states that the $L$-function of an overconvergent $\sigma$-module is a meromorphic function on the \emph{relative} rigid affine line $\A^{1,\rig}_\calW$. Essentially, this means that for any point $P$ in the weight space, the specialization of $L(\phi,s)$ at $P$ is a $p$-adic meromorphic function in the usual sense.
	
	In addition to this result, a major motivation in writing this article was to develop the analogy with the theory of Coleman-Mazur. In our setting, the role of their overconvergent modular forms is played by coherent modules on a \emph{weak rigid analytic space} over $R$. The first two parts of this paper are devoted to developing the theory of weak analytic geometry relative to $R$. In the classical case that $R$ is a discrete valuation ring, this reduces to the theory of dagger spaces over the field of fractions $Q(R)$, as introduced in \cite{Grosse}. Langer and Muralidharan have shown that these spaces arise as the generic fibers of weak formal $R$-schemes \cite{Langer}. Our construction follows theirs closely: Let $\FS_R^\dagger$ denote the category of weak formal schemes over $R$. The category $\Rig_R^\dagger$ of weak rigid analytic spaces over $R$ can be constructed by localizing this category at a suitable class of morphisms known as \emph{admissible blow-ups}. We may interpret such a space as a continuous family of dagger spaces, parameterized by the points in the weight space $\calW$. We equip each weak rigid analytic space with the structure of a locally ringed topos. If $\frakX$ is an object of $\FS_R^\dagger$ and $\calX$ is its ``generic fiber,'' then we attach to each coherent sheaf $\calF$ of $\calO_\frakX$-modules a coherent sheaf $\calF^\rig$ of $\calO_\calX$-modules referred to as the \emph{analytification} of $\calF$. Passing to completions, we recover the classical generic fiber functor as developed by Raynaud \cite{Raynaud}, and Bosch and L\"{u}tkebohmert \cite{Bosch} for more general $R$.
	
	The connection between our theory and the theory of weak analytic geometry over $R$ is given as follows: Let $(M,\phi)$ be an overconvergent $\sigma$-module and let $M^\vee$ denote the dual module. There is a natural duality $\phi \mapsto \theta_\bullet(\phi)$ which assigns to $\phi$ a sequence of \emph{Dwork operators} on the de Rham complex $\Omega^\bullet M^\vee$. The essential feature of these operators is that the action of $\theta_i$ on $(\Omega^i M)^\rig$ is in some sense \emph{nuclear}. In particular, each $\theta_i(\phi)$ has a well defined \emph{Fredholm determinant} $C(\theta_i(\phi),s)$ which we may regard as an analytic function on $\A^{1,\rig}_\calW$. The full form of our main theorem is the following ``trace formula:''
	
	\begin{theorem}\label{t:merom}
		Let $(M,\phi)$ be an overconvergent $\sigma$-module over $X$. Then
		\begin{equation*}
			L(\phi,s)	=	\prod_{i=0}^n	C(\theta_{n-i}(\phi),s)^{(-1)^{i-1}}.
		\end{equation*}
	\end{theorem}
	
	Our trace formula reduces the study of the $L$-function $L(\phi,s)$ to understanding the spectral theory of nuclear operators over $\calW$, which we discuss in \ref{s:eig}. Given a nuclear operator $\psi$, its spectral theory is best described in terms of its \emph{eigenvariety} $\calE(\psi)$, which is defined to be the rigid analytic hypersurface cut out by the Fredholm determinant $C(\psi,s)$. If $P$ is a point in the weight space $\calW$, then the fiber $\calE(\psi)_P$ is naturally identified with the set of non-zero eigenvalues of $\psi$ acting on a vector space over the residue field $k(P)$. It is natural to ask how this set of eigenvalues varies $p$-adically in the parameter $P$. In \cite{Davis}, the spectral halo property of their eigencurve answers this question completely: for each $P$, the $p$-adic valuations of the points in $\calE(\psi)_P$ form finite (fixed) number of arithmetic progressions, each of whose increment is a scalar multiple of the $p$-adic valuation of their weight parameter. Via the trace formula, this gives a striking uniformity in the distribution of zeros for Artin-Schreier-Witt coverings of the affine line.
	
	As an application, in Section \ref{s:hodge} we apply Theorem \ref{t:merom} to the study Artin-Schreier-Witt coverings of our smooth affine variety $X$. We attach to each such tower a family of exponential sums taking values in the ring $R=\Z_p[[T]]$, generalizing the \emph{$T$-adic exponential sums} of Liu and Wan \cite{Liu} to general smooth affine varieties. We define the \emph{$T$-adic $L$-function} $L(T,s) \in R[[s]]$ to be the generating function of these exponential sums. We show that $L(T,s)$ agrees with the $L$-function of a certain rank-$1$ character $\pi_1(X,\bar{x}) \to R^\times$. When the corresponding $\sigma$-module $(M,\phi)$ over $\frakX_\infty$ is overconvergent, Theorem \ref{t:merom} gives an explicit trace formula for our $T$-adic exponential sums, and consequently for the zeta functions of Artin-Schreier-Witt coverings of $X$. We do not know whether the spectral halo property of \cite{Davis} holds in this general setting. It is likely that an answer to this question requires some understanding of the local monodromy of overconvergent $\sigma$-modules over $R$, a topic which we hope to address in a future article.
	
\section{Weak Formal Geometry}
	
	\subsection{Formal Schemes and Rigid Analytic Spaces}\label{s:formal}
	
		In this preliminary section we recall some basic notions of formal schemes and their associated rigid analytic spaces. As is common in formal geometry, we will restrict our attention to a small class of well behaved formal schemes. Let $\FS^+$ denote the category of locally Noetherian formal schemes which are complete and separated with respect to an ideal of definition. Recall that a morphism $f:\frakY \to \frakX$ in $\FS^+$ is \emph{adic} if $\frakX$ admits an ideal of definition $\calI$ such that $f^* \calI \calO_\frakY$ is an ideal of definition of $\frakY$. Since $\frakX$ is locally Noetherian, every ideal of definition is finitely generated, and it follows that $\calI_\frakY = f^* \calI \calO_\frakY$ is always ideal of definition of $\frakY$. Moreover, there is a cartesian diagram of ringed spaces (\cite{Abbes}, 2.2.9):
		\begin{equation*}
			\begin{tikzcd}
				(\frakY,\calO_\frakY/\calI_\frakY)	\arrow[d]	\arrow[r,"f'"]	&	(\frakX,\calO_\frakX/\calI)	\arrow[d]	\\
				(\frakY,\calO_\frakY)	\arrow[r,"f"]	&	(\frakX,\calO_\frakX)
			\end{tikzcd}
		\end{equation*}
		We say that $f$ is \emph{locally of finite presentation} if $f'$ is is locally of finite presentation (as a morphism of schemes). As usual, we say that $f$ is of \emph{finite presentation} if $f$ is locally of finite presentation and quasi-compact. We denote by $\FS$ the subcategory of $\FS^+$ with the same objects, but whose morphisms are adic morphisms of finite presentation.
		
		It is often convenient to think of an object of $\FS^+$ as (locally) the formal completion of an ordinary scheme. Let $X$ be a locally Noetherian scheme, and $\calI$ a sheaf of ideals on $X$. Then the $\calI$-adic completion $X_\infty$ is an object of $\FS^+$. Let $\iota:X_\infty \to X$ denote the corresponding map of ringed spaces. The $\calI$-adic completion of modules defines a functor
		\begin{equation*}
			\Mod(\calO_X)	\to	\Mod(\calO_{X_\infty}).
		\end{equation*}
		which we denote by $\calF \mapsto \calF_\infty$. We will mainly be interested in the case that $\calF$ is a finite $\calO_X$-module, in which case we have $\calF_\infty = \iota^*\calF$. If $X=\Spec(A)$ is affine and $M$ is an $A$-module, we write $\tilde{M}$ for the sheaf of $\calO_X$-modules associated to $M$.
		
		\begin{proposition}\label{p:coh}
			Let $\frakX$ be an object of $\FS^+$, and $\calI$ an ideal of definition of $\frakX$. The following are equivalent for a sheaf $\calF$ of $\calO_\frakX$-modules:
			\begin{enumerate}
				\item	$\calF$ is coherent
				\item	$\calF$ is of finite type
				\item	For every affine open $\frakU=\Spf(A)$ of $\frakX$, there is a finite $A$-module such that
					\begin{equation*}
						\calF|_\frakU	=	\tilde{M}_\infty.
					\end{equation*}
			\end{enumerate}
		\end{proposition}
		\begin{proof}
			See (\cite{Abbes}, 2.7.2 and 2.8.2).
		\end{proof}
		
		In this article we will typically work relative to a base ring $R$, which we assume to be a complete Noetherian local domain with maximal ideal $\frakm$. Let $\FS_R$ denote the category of morphisms $\frakX \to \Spf(R)$ in $\FS$. All objects of $\FS_R$ carry the $\frakm$-adic topology. Given an object $\frakX$ of $\FS_R$, we may write $\frakX_n$ for the reduction of $\frakX$ mod $\frakm^{n+1}$, so that $\frakX = \varinjlim_n \frakX_n$. If $\Fin_R$ denotes the category of $R$-schemes locally of finite presentation, then $\frakm$-adic completion induces a functor
		\begin{equation*}
			\Fin_R	\to	\FS_R.
		\end{equation*}
		Suppose that $h:R \to S$ is a map of Noetherian local domains. To every object $\frakX$ of $\FS_R$, we define the \emph{formal base change}
		\begin{equation*}
			\frakX_h	=	\varinjlim_n	(\frakX \times_R S)_n.
		\end{equation*}
		Then $\frakX_h$ is an object of $\FS_S$, and there is an induced morphism $\frakX_h \to \frakX$ in $\FS^+$. The assignment $\frakX \to \frakX_h$ gives a formal base change functor
		\begin{equation*}
			\FS_R	\to	\FS_S.
		\end{equation*}
		
		Our notion of rigid analytic spaces differs from the usual notion, in particular we do not assume that a rigid analytic space is defined relative to a field. Rather, we may construct a category $\Rig$ of quasi-compact, quasi-separated rigid analytic spaces by localizing the category $\FS$ at a suitable class of morphisms called \emph{admissible blow-ups}. Let $R$ be as above, and $\calW$ the rigid analytic space associated to $\Spf(R)$. The category $\Rig_R$ is defined to be the category of morphisms $\calX \to \calW$ in $\Rig$. When $R$ is a discrete valuation ring, the localization functor $\FS_R \to \Rig_R$ coincides with Raynaud's generic fiber functor \cite{Raynaud}.
		
		For the study of eigenvarieties, we will need to enlarge the category $\Rig$ to include rigid analytic spaces which are not necessarily quasi-compact. Recall that an \emph{ind-object} in $\Rig$ is a functor $I \to \Rig$, where $I$ is a small filtered category. We let $\underline{\Rig}$ denote the category of ind-objects $I \to \Rig$ such that every morphism in $I$ maps to an open immersion in $\Rig$. If $\calX:I \to \Rig$ and $\calY:J \to \Rig$ are two objects of $\underline{\Rig}$, then we have
		\begin{equation*}
			\Hom_{\underline{\Rig}}(\calX,\calY)	=	\varprojlim_{i \in I} \varinjlim_{j \in J}	\Hom_{\Rig}(\calX_i,\calY_j).
		\end{equation*}
		The objects of $\underline{\Rig}$ may be regarded as directed unions of quasi-compact rigid analytic spaces. Our basic examples of such spaces will be the relative affine line $\A^1_\calW$ and the relative multiplicative group $\G_{m,\calW}$.
		
		For convenience of the reader, we develop the notions of weak and ordinary analytic geometry over $R$ in parallel. In proofs, we will usually restrict our attention to the weak case. Often, our proofs adapt readily to the ordinary case by replacing the notion of ``weak completion'' with $\frakm$-adic completion. Alternatively, all statements for rigid analytic spaces in this sense can be found in \cite{Abbes}.
		
	\subsection{Weak Formal Schemes}
	
		The notion of formal schemes with a weakly complete structure sheaf has been developed by Meredith \cite{Meredith}. This section is intended as a brief overview of the theory. We begin by recalling the overconvergent theory of Monsky and Washnitzer. For $r \in R$, let
		\begin{equation*}
			v_\frakm(r)	=	\sup\left\{	i:r \in \frakm^i	\right\}.
		\end{equation*}
		We say that a power series $f = \sum_u f_u X^u \in R[[X_1,...,X_n]]$ is \emph{overconvergent} if
		\begin{equation*}
			\lim_{k \to \infty}	\inf_{|u| > k}	\frac{v_\frakm(f_u)}{|u|}	>	0.
		\end{equation*}
		The overconvergent power series in $n$ variables form an $R$-algebra which we denote by $R[X_1,...,X_n]^\dagger$. It is well known that this ring is Noetherian \cite{Fulton}. For any $R$-algebra $A$, we write $A_\infty$ for its $\frakm$-adic completion. The \emph{weak completion} of $A$ is the $R$-subalgebra $A^\dagger$ of $A_\infty$ consisting of elements of the form $f(a_1,...,a_n)$, where $f \in R[X_1,...,X_n]^\dagger$ for some $n$, and $a_1,...,a_n \in R$. The canonical map $A \to A^\dagger$ is flat provided that $A$ is Noetherian (\cite{Meredith}, 1.3). We say that $A$ is \emph{weakly complete} if this map is an isomorphism. We will frequently make use of the following result of Monsky and Washnitzer (\cite{MW}, 1.6):
		
		\begin{lemma}\label{l:ff}
			If $A$ is weakly complete, then $A$ is a Zariski ring. In particular, the $\frakm$-adic completion $A \to A_\infty$ is faithfully flat.
		\end{lemma}
		
		\begin{definition}
			An $R$-algebra $A$ is a w.c.f.g. algebra (``weakly complete and weakly finitely generated'') if $A$ is a quotient of $R[X_1,...,X_n]^\dagger$ for some $n$.
		\end{definition}
		
		Let $A$ be a w.c.f.g. algebra, and $A_0=A/\frakm A$. For each $f \in A$, the map $A \to A_f$ is flat of finite presentation. In particular, $A \to A_f^\dagger$ is a flat map of w.c.f.g. algebras. If $M$ is an $A$-module of finite type, we define $M_f^\dagger = M \otimes_A A_f^\dagger$. Let $A_{f,0} = A_f/\frakm A_f$. By (\cite{Meredith}, 2.3), if $\Spec(A_{f,0}) \supseteq \Spec(A_{g,0})$, then the map $A_{f,0} \to A_{g,0}$ lifts uniquely to a map $A_f^\dagger \to A_g^\dagger$ of $A$-algebras. In particular, we can associate to $M$ a presheaf $\tilde{M}^\dagger$ on the distinguished open sets of $A_0$ by setting
		\begin{equation*}
			\Gamma( \Spec(A_{f,0}),\tilde{M}^\dagger)	=	M_f^\dagger.
		\end{equation*}
		The main Theorem of (\cite{Meredith}, \S 2) then states:
		
		\begin{theorem}
			(Meredith). The presheaf $\tilde{M}^\dagger$ is a sheaf.
		\end{theorem}
		
		Consequently, $\tilde{M}^\dagger$ prolongs uniquely to a Zariski sheaf on $\Spec(A_0)$. We define the \emph{weak formal spectrum} of $A$ to be the locally ringed space $\Spwf(A) = (\Spec(A_0),\tilde{A}^\dagger)$.
		
		\begin{definition}
			A locally ringed space is an \emph{affine weak formal scheme} over $R$ if it is isomorphic to $\Spwf(A)$ for some w.c.f.g algebra $A$ over $R$. A \emph{weak formal scheme} is a locally ringed space which admits a covering by affine weak formal schemes.
		\end{definition}
		
		It follows from the definition that every weak formal scheme over $R$ is locally Noetherian. We remark that the ringed spaces $\frakW = \Spf(R)$ and $\Spwf(R)$ are isomorphic, and it will be convenient to regard $\frakW$ as both a weak and ordinary formal scheme. We denote by $\FS_R^\dagger$ the category of quasi-compact weak formal schemes over $R$. If $\frakX$ is an object of $\FS_R^\dagger$, then the $\frakm$-adic completion $\frakX_\infty$ is an object of $\FS_R$. By Lemma \ref{l:ff}, the induced functor
		\begin{equation*}
			\Mod(\calO_\frakX)	\to	\Mod(\calO_{\frakX_\infty})
		\end{equation*}
		is faithfully exact. Our next goal is to show that the $\frakm$-adic completion functor $\Fin_R \to \FS_R$ factors through $\FS_R^\dagger$ by defining the \emph{weak completion} of an $R$-scheme locally of finite presentation. First, we state an easy ``gluing lemma'' which will be useful for globalizing constructions of w.c.f.g. algebras (\cite{Langer}, 2.7):
		
		\begin{lemma}\label{l:glue}
			Let $X$ be a topological space and $\calG$ a sheaf on $X$. Let $\{X_i\}$ be an open covering of $X$, and $\calF_i$ a subsheaf of $\calG|_{X_i}$ satisfying $\calF_i|_{X_i \cap X_j} = \calF_j|_{X_i \cap X_j}$. Then there is a unique minimal subsheaf $\calF$ such that $\calF|_{X_i} = \calF_i$ for all $i$.
		\end{lemma}
		\begin{proof}
			For each open subset $U \subseteq X$, we let
			\begin{equation*}
				\calF(U)	=	\{s \in \calG(U): s|_{X_i \cap U} \in \calG(X_i \cap U)\text{ for all }i\}.
			\end{equation*}
			It is not difficult to see that $\calF$ is a sheaf on $X$. If $\calF'$ is a sheaf satisfying $\calF'|_{X_i} = \calF_i$ for all $i$, then by the sheaf axiom there is an inclusion $\calF(U) \to \calF'(U)$ for all open $U \subseteq X$.
		\end{proof}
		
		\begin{proposition}\label{p:wcomp}
			The $\frakm$-adic completion functor $\Fin_R \to \FS_R$ factors uniquely through a functor $\Fin_R \to \FS_R^\dagger$ 
		\end{proposition}
		\begin{proof}
			Let $X$ be a scheme locally of finite presentation over $R$. We construct an object $X^\dagger$ of $\FS_R^\dagger$ as follows: the underlying topological space of $X^\dagger$ is that of the $\frakm$-adic completion $X_\infty$. To define the structure sheaf, let $\{X_i \to X\}_i$ be an open affine cover with $X_i = \Spec(A_i)$.  Then $\{(X_i)_\infty \to X_\infty \}_i$ is an open cover of $X_\infty$. Let $\calF_i = \tilde{A}_i^\dagger$, which is a sheaf of w.c.f.g. algebras on $(X_i)_\infty$. To see that these $\calF_i$ satisfy the conditions of Lemma \ref{l:glue}, note that the $\frakm$-adic completions $(\calF_i)_\infty$ agree on double intersections, and so the claim follows from Lemma \ref{l:ff}. We define $\calO_X^\dagger$ to be resulting sheaf of w.c.f.g. algebras on $X^\dagger$. The ringed space $(X^\dagger,\calO_X^\dagger)$ is an object of $\FS_R^\dagger$.
			
			If $f:Y \to X$ is a morphism in $\Fin_R$, for each $i$ choose an affine open cover $\{f_{i,j}:Y_{i,j} \to f^{-1}(X_i)\}_j$ with $Y_{i,j} = \Spec(B_{i,j})$. By (\cite{MW}, 1.5), there are unique maps of w.c.f.g. algebras $A_i^\dagger \to B_{i,j}^\dagger$ extending the maps $A_i \to B_{i,j}$. Let $f_{i,j}^\dagger$ be the corresponding map of weak formal schemes. To see that the morphisms $Y_{i,j}^\dagger \to X_i^\dagger \to X^\dagger$ agree on double intersections, we may simply note that the statement is true on the level of $\frakm$-adic completions, and again apply Lemma \ref{l:ff}.
		\end{proof}
		
		We refer to the functor of Proposition \ref{p:wcomp} as \emph{weak completion}. Given an object $X$ of $\Fin_R$, let $\iota^\dagger:X^\dagger \to X$ denote the corresponding map of ringed spaces. The assignment $\calF \mapsto \calF^\dagger = (\iota^\dagger)^* \calF$ defines a functor
		\begin{equation*}
			\Mod(\calO_X)	\to	\Mod(\calO_{X^\dagger}).
		\end{equation*}
		which is \emph{faithfully} exact, by Lemma \ref{l:ff}. When $\calF$ is a finite $\calO_X$-module, $\calF^\dagger$ may be regarded as the \emph{weak completion} of $\calF$. We have the following characterization of coherent modules on a weak formal scheme:
		
		\begin{proposition}\label{p:wcoh}
			Let $\frakX$ be an object of $\FS_R^\dagger$. The following are equivalent for a sheaf $\calF$ of $\calO_\frakX$-modules:
			\begin{enumerate}
				\item	$\calF$ is coherent
				\item	$\calF$ is of finite type
				\item	For every affine open $\frakU=\Spwf(A)$ of $\frakX$, there is a finite $A$-module such that
					\begin{equation*}
						\calF|_\frakU	=	\tilde{M}^\dagger.
					\end{equation*}
			\end{enumerate}
		\end{proposition}
		\begin{proof}
			Since the $\frakm$-adic completion functor $\Mod(\calO_\frakX) \to \Mod(\calO_{\frakX_\infty})$ is exact, Proposition \ref{p:coh} immediately gives (iii) $\Rightarrow$ (i) $\Leftrightarrow$ (ii). To see (i) $\Rightarrow$ (iii), we may reduce to the case that $\frakX=\Spwf(A)$ is affine. Let $M = \Gamma(\frakX,\calF)$, which is a finite $A$-module. There is a canonical map $\tilde{M} \to \calF$, inducing a map $\tilde{M}^\dagger \to \calF^\dagger = \calF$. By Proposition \ref{p:coh}, this is an isomorphism on the level of $\frakm$-adic completions. The result follows from the fact that the $\frakm$-adic completion of $\calO_\frakX$-modules is faithfully exact.
		\end{proof}
		
		The weak completion construction can be used to show e.g. that $\FS^\dagger$ has finite limits and gluings. As in the case of formal schemes, we can also consider the base change of weak formal schemes:
		
		\begin{lemma}\label{l:wbchange}
			Let $h:R \to S$ be a continuous map of complete Noetherian local rings, and $A$ a w.c.f.g. algebra over $R$. Let $A \otimes_R^\dagger S$ denote the weak completion of $A \otimes_R S$ as an $S$-algebra. Then $A \otimes_R^\dagger S$ is a w.c.f.g. algebra over $S$.
		\end{lemma}
		\begin{proof}
			Choose a presentation $A = R_1[X_1,...,X_n]/\fraka$. Since $A \otimes_R^\dagger S$ is weakly complete, there is a unique $S$-algebra map
			\begin{equation*}
				S[X_1,...,X_n]^\dagger	\to	A \otimes_R^\dagger S
			\end{equation*}
			sending $X_i \mapsto X_i \otimes 1$. This map is clearly surjective, thus giving the result.
		\end{proof}
		
		\begin{proposition}\label{p:wbchange}
			Let $h:R \to S$ as in Lemma \ref{l:wbchange}. There is a unique \emph{weak base change} functor $\FS_R^\dagger \to \FS_S^\dagger$ compatible with $\frakm$-adic completion, in the sense that the following diagram commutes:
			\begin{equation*}
				\begin{tikzcd}
					\FS_R^\dagger	\arrow[d]	\arrow[r]	&	\FS_S^\dagger	\arrow[d]	\\
					\FS_R	\arrow[r]	&	\FS_S
				\end{tikzcd}
			\end{equation*}
		\end{proposition}
		\begin{proof}
			Let $\frakX$ be an object of $\FS_R^\dagger$. Choose a covering $\{\frakX_i \to \frakX\}_i$ where $\frakX_i = \Spwf(A_i)$ for some w.c.f.g. algebra $A_i$ over $R$. Let $\frakX_{i,j} = \frakX_i \cap \frakX_j$. Using Lemma \ref{l:wbchange}, we obtain a weak affine formal scheme $(\frakX_i)_h=\Spf(A_i \otimes_R^\dagger S)$ over $S$. For each $j$, there is an open immersion $(\frakX_{i,j})_h	\to (\frakX_i)_h$, and the identity map $\frakX_{i,j} \to \frakX_{j,i}$ induces an isomorphism $(\frakX_{i,j})_h \to (\frakX_{j,i})_h$ (indeed, by Lemma \ref{l:ff}, this can be checked on the level of completions). Gluing the $(\frakX_i)_h$, we obtain the desired formal $S$-scheme $\frakX_h$.
		\end{proof}
		
		In the notation of the proof, there is a natural map of ringed spaces $\frakX_h \to \frakX$. Proposition \ref{p:wcoh} immediately gives the following:
		
		\begin{proposition}\label{p:cohbchange}
			Let $h:R \to S$ as in Lemma \ref{l:wbchange}, and let $\frakX$ be an object of $\FS_R^\dagger$. For every coherent $\calO_\frakX$-module $\calF$, the $\calO_{\frakX_h}$-module
			\begin{equation*}
				\calF_h	=	\calF	\otimes_{\calO_\frakX}	\calO_{\frakX_h}
			\end{equation*}
			is coherent.
		\end{proposition}
		
	\subsection{Admissible Blowing-Up}
	
		Henceforth we will let $\FS_R^*$ denote either $\FS_R$ or $\FS_R^\dagger$. We will construct the corresponding category $\Rig_R^*$ of \emph{(weak) rigid analytic spaces} over $R$ by localizing $\FS_R^*$ at a class of morphisms referred to as \emph{admissible blow-ups}. We begin by recalling the definition for ordinary formal schemes. Let $\frakX$ be an object of $\FS_R$, and $\calI$ an open ideal of $\calO_\frakX$. The \emph{admissible formal blow-up} of $\frakX$ along $\calI$ is the formal $R$-scheme
		\begin{equation*}
			\frakX_\calI	=	\varinjlim_n	\Proj\left(	\bigoplus_d	\calI^d	\otimes_{\calO_\frakX}	\calO_{\frakX_n}	\right)
		\end{equation*}
		equipped with the canonical projection $\frakX_\calI \to \frakX$. More generally, a morphism $\frakX' \to \frakX$ of formal $R$-schemes is an \emph{admissible formal blow-up} if there exists an open ideal $\calI$ on $\frakX$ such that $\frakX'$ is $\frakX$-isomorphic to $\frakX_\calI$. The following lemma indicates that admissible formal blow-ups are locally the $\frakm$-adic completion of a scheme-theoretic blow up (\cite{Abbes}, 3.13):
		
		\begin{lemma}
			Let $\frakX$ be an object of $\FS_R$, and $\calI$ an open ideal of $\calO_\frakX$.
			\begin{enumerate}
				\item	For every open formal subscheme $\frakU$ of $\frakX$, $\frakX_\calI \times_\frakX \frakU \to \frakU$ is the admissible blow-up of $\frakU$ at $\calI|_\frakU$.
				\item	Suppose that $\frakX=\Spf(A)$ and $I$ is an open ideal of $A$ such that $\calI = \tilde{I}$. Let $X = \Spec(A)$. Then $\frakX_\calI$ is the $\frakm$-adic completion of the scheme-theoretic blow-up $X_I \to X$.
			\end{enumerate}
		\end{lemma}
		
		\begin{proposition}\label{p:wbl}
			Let $\frakX$ be an object of $\FS_R^\dagger$, and $\calI$ an open ideal of $\calO_\frakX$. There is a unique map of weak formal $R$-schemes $\frakX_\calI \to \frakX$ whose $\frakm$-adic completion is the admissible formal blow-up of $(\frakX_\infty)_{\calI_\infty} \to \frakX_\infty$.
		\end{proposition}
		\begin{proof}
			The underlying map of topological spaces $\frakX_\calI \to \frakX$ is simply that of the admissible formal blow-up $(\frakX_\infty)_{\calI_\infty} \to \frakX_\infty$. To construct the structure sheaf, we use Lemma \ref{l:glue}. Let $\calG$ be the structure sheaf of $(\frakX_\infty)_{\calI_\infty}$. Choose an open cover $\{\frakX_i\}$ of $\frakX$ such that $\frakX_i=\Spwf(A_i)$, where $A_i$ is a w.c.f.g algebra, and $\calI|_{\frakX_i} = \tilde{I}_i^\dagger$ for some open ideal $I_i$ of $A_i$. Write $X_i = \Spec(A_i)$. We define the admissible weak blow-up of $\frakX_i$ to be the weak completion of the ordinary blow-up $(X_i)_{I_i} \to X_i$. Denote this map by $\frakU_i \to \frakX_i$. The $\frakU_i$ form an open cover of $\frakX_\calI$, and the structure sheaf $\calF_i$ of $U_i$ is a subsheaf of $\calG|_{\frakU_i}$. Again by Lemma \ref{l:ff}, it we see easily that the conditions of Lemma \ref{l:glue} are satisfied, and we let $\calO_{\frakX_\calI}$ be the corresponding sheaf on $\frakX_\calI$. Then $\frakX_\calI$ is a weak formal scheme by construction, and we have a natural map of weak formal $R$-schemes $\frakX_\calI \to \frakX$.
		\end{proof}
		
		We refer to the map $\frakX_\calI \to \frakX$ of Proposition \ref{p:wbl} as the \emph{admissible weak blow-up} of $\frakX$ at $\calI$. It is not immediately clear that this map is a morphism in $\FS_R^\dagger$. However, this can be seen from the following description of admissible weak blow-ups in coordinates:
		
		\begin{proposition}\label{p:coords}
			Let $\frakX=\Spwf(A)$ for some w.c.f.g. algebra $A$, and $I=(a_1,...,a_n)$ an open ideal of $A$. Write $\calI = \tilde{I}^\dagger$. For each $i$, define
			\begin{equation*}
				A_i	=	A\left[	\frac{a_j}{a_i}:j \neq i	\right].
			\end{equation*}
			Then the weak affine formal schemes $X_i = \Spwf(A_i^\dagger/(\text{$a_i$-tor}))$ form an affine open cover of $\frakX_\calI$. If $\frakX$ is $\frakm$-torsion free, then $X_i = \Spwf(A_i^\dagger/(\text{$\frakm$-tor}))$ and $\frakX'$ is $\frakm$-torsion free as well.
		\end{proposition}
		\begin{proof}
			Let $X = \Spec(A)$. Then from the theory of ordinary blow-ups, $X_I$ is covered by open affine subsets of the form $\Spec(A_i/(\text{$a_i$-tor}))$. Passing to weak completions, we see that $\frakX_\calI$ is covered by weak affine formal schemes of the form $\Spwf( (A_i/(\text{$a_i$-tor}))^\dagger)$. Since $A_i \to A_i^\dagger$ is flat, $(A_i/(\text{$a_i$-tor}))^\dagger = A_i^\dagger/(\text{$a_i$-tor})$ as desired. To see the second statement, note that $I A_i=(a_i)$ is an open ideal, so that $(\text{$a_i$-tor}) \subseteq (\text{$\frakm$-tor})$ in $A_i$. But if $A$ is $\frakm$-torsion free, so is the graded ring $\bigoplus_d I^d$ and its localizations. It follows that this containment is an equality. The result follows again by flatness of $A_i \to A_i^\dagger$.
		\end{proof}
		
		\begin{proposition}\label{p:univ}
			Let $\frakX$ be an object of $\FS_R^*$ and $\calI$ an open ideal of $\calO_\frakX$. The admissible blow-up satisfies the following universal property: every morphism $\pi:\frakY \to \frakX$ for which $\pi^{-1}\calI \calO_\frakY$ is an invertible sheaf factors uniquely through $\frakX_\calI \to \frakX$. If $\pi$ is a morphism in $\FS_R$, then so is the map $\frakY \to \frakX_\calI$.
		\end{proposition}
		\begin{proof}
			The problem is local on $\frakX$ and $\frakY$, so assume that $\frakX=\Spwf(A)$, $\calI = \tilde{I}^\dagger$, and $\frakY=\Spwf(B)$. Choose a generating set $I=(a_1,...,a_n)$, and let $A_i$, $\frakX_i$ be as in the proof of Proposition \ref{p:coords}. The ideal $IB$ is invertible, say $IB = (a_i)B$. There is a unique map
			\begin{equation*}
				A_i^\dagger/(\text{$a_i$-tor})	\to	B
			\end{equation*}
			extending the given map $A \to B$. Gluing these maps gives the desired map $\frakY \to \frakX_\calI$. To see uniqueness, note that for every map $\frakY \to \frakX_\calI$ with $\pi^{-1}\calI \calO_\frakY$ generated by $a_i$, the image of $\frakY$ must lie in $\frakX_i$. But we have just seen that there is a unique such map.
		\end{proof}
		
		At this point, we may define the category $\Rig_R^*$ to be the localization of $\FS_R^*$ at the class of admissible blow-ups. We may also define the absolute category $\Rig$ to be the localization of $\FS$ at the class of admissible formal blow-ups, however there is no analogue of this category in the weakly complete case. We will denote the localization functor by $\frakX \mapsto \frakX^\rig$ and similarly for morphisms. A \emph{model} of an object $\calX$ in $\Rig_R^*$ is defined to be an object $\frakX$ in $\FS_R^*$ such that $\frakX^\rig \cong \calX$. A model of a morphism $\calX \to \calY$ in $\Rig_R^*$ is a morphism $\frakX \to \frakY$ in $\FS_R^*$ fitting into a commutative diagram
		\begin{equation*}
			\begin{tikzcd}
				\frakX^\rig	\arrow[d]	\arrow[r]	&	\frakY^\rig	\arrow[d]	\\
				\calX	\arrow[r]	&	\calY
			\end{tikzcd}
		\end{equation*}
		where the vertical arrows are isomorphisms. We now consider some good categorical properties of the admissible blow-ups in $\FS_R^*$:
		
		\begin{lemma}\label{l:comp}
			Let $\frakX'' \to \frakX'$ and $\frakX' \to \frakX$ be two admissible blow-ups in $\FS_R^*$. Then the composition $\frakX'' \to \frakX$ is an admissible blow-up.
		\end{lemma}
		\begin{proof}
			Suppose first that $\frakX=\Spwf(A)$ is affine, and that $\frakX' \to \frakX$ is the admissible blow-up of $\frakX$ at $\tilde{I}^\dagger$ for some open ideal $I$ of $A$. Let $\calJ$ be an open sheaf of ideals on $\frakX'$ for which $\frakX'' \to \frakX'$ is the weak admissible blow-up of $\frakX'$ at $\calJ$. Let $X=\Spec(A)$, so that $\frakX' \to \frakX$ is the weak completion of the ordinary blow-up $X_I \to X$. We claim $\calJ$ is the weak completion of an open sheaf of ideals $\calJ'$ on $X_I$. To see this, we will apply Lemma \ref{l:glue}. Choose a covering $\{V_j\}$ of $X_I$ with $V_j = \Spec(A_j')$. Then $\{V_j^\dagger\}$ is an open covering of $\frakX'$ with $V_j^\dagger = \Spwf(A_j^\dagger)$, and $\calJ|_{V_j^\dagger} = \tilde{J}_j^\dagger$ for some open ideal $J_j$ of $A_j^\dagger$. Let $J_j' = J_j \cap A_j$, and $\calJ_j' = \tilde{J}_j'$ the associated sheaf of ideals on $V_j$. Then by passing to $\frakm$-adic completions and using Lemma \ref{l:ff}, it follows that the collection $\calJ_j'$ satisfy the conditions of Lemma \ref{l:glue}, giving the desired sheaf $\calJ'$ on $X_I$. The composition
			\begin{equation*}
				(X_I)_{\calJ'}	\to	X_I	\to	X
			\end{equation*}
			is a composition of (scheme-theoretic) blow-ups and therefore is the blow-up of $X$ at some open ideal $I'$. To see the general case, recall that the blow-up of schemes commutes with flat base change, and therefore so does the construction of the ideal $I'$. Using Lemma \ref{l:glue}, we may cover $\frakX$ by weak affine formal schemes and glue the resulting $(I')^\dagger$ to obtain an open ideal $\calI$ of $\calO_\frakX$. We may check locally that the inverse image of $\calI$ in $\frakX''$ is an invertible sheaf, and therefore we have an $\frakX$-morphism $\frakX'' \to \frakX_\calI$. By the preceding local construction, this map is in fact an isomorphism.
		\end{proof}
		
		\begin{lemma}\label{l:bchange}
			Let $f:\frakX \to \frakY$ be a morphism in $\FS_R^*$ and $\frakY' \to \frakY$ an admissible blow-up. Then there exists a diagram in $\FS_R^*$
			\begin{equation*}
				\begin{tikzcd}
					\frakX'	\arrow[bend right,ddr,"\varphi"]	\arrow[bend left,drr]	\arrow[dr,"\psi"]	&	&	\\
						&	\frakX \times_\frakY \frakY'	\arrow[d,"\pi"]	\arrow[r]	&	\frakY'	\arrow[d]	\\
						&	\frakX	\arrow[r,"f"]	&	\frakY
				\end{tikzcd}
			\end{equation*}
			where $\varphi$ and $\psi$ are admissible blow-ups.
		\end{lemma}
		\begin{proof}
			Say $\frakY' \to \frakY$ is the admissible blow-up of $\frakY$ at the open ideal $\calJ$. Let $\calI = f^{-1}\calI \calO_\frakX$, which is an open ideal of $\calO_\frakX$. Let $\frakX' = \frakX_\calI$. The inverse image of $\calJ$ under the composition $\frakX' \to \frakX \to \frakY$ generates an invertible ideal of $\frakX'$, therefore by Proposition \ref{p:univ}, we have a map $\frakX' \to \frakY'$. By the universal property of the fiber product, there is a unique map $\psi:\frakX' \to \frakX \times_\frakY \frakY'$ making the diagram commute. Note however that the admissible blow-up of $\frakX \times_\frakY \frakY'$ at $\pi^{-1}\calI \calO_{\frakX \times_\frakY \frakY'}$ also satisfies this property.
		\end{proof}
		
		\begin{lemma}\label{l:coeq}
			Suppose that
			\begin{equation*}
				\begin{tikzcd}
					\frakX	\arrow[r,shift left] \arrow[r,shift right] &	\frakY'	\arrow[r,"\psi"]	&	\frakY
				\end{tikzcd}
			\end{equation*}
			is a commutative diagram in $\FS_R^*$, where $\psi$ is an admissible blow-up. Then there is an admissible blow-up $\varphi:\frakX' \to \frakX$ making the diagram
			\begin{equation*}
				\begin{tikzcd}
					\frakX'	\arrow[r,"\varphi"]	&	\frakX	\arrow[r,shift left] \arrow[r,shift right] &	\frakY'
				\end{tikzcd}
			\end{equation*}
			commute.
		\end{lemma}
		\begin{proof}
			Say $\psi$ is the admissible blow-up of $\frakY$ at an open ideal $\calJ$. Let $\calI$ be the open ideal of $\calO_\frakX$ generated by the inverse image of $\calJ$ along $\frakX \to \frakY$. Then we can define define $\varphi$ to be the admissible blow-up of $\frakX$ at $\calI$. Both arrows $\frakX \to \frakY'$ are the admissible blow-up of $\frakY'$ at $\psi^{-1} \calJ \calO_{\frakY'}$.
		\end{proof}
		
		Lemmas \ref{l:comp}-\ref{l:coeq} together imply that the class of admissible blow-ups in $\FS_R^*$ constitutes a calculus of left fractions. In particular, a morphism $\frakX^\rig \to \frakY^\rig$ may be represented by a diagram $\frakX \xleftarrow{\varphi} \frakX' \to \frakY$, where $\varphi$ is an admissible blow-up if $\FS_R$. These diagrams are subject to the usual equivalence relation which allows one to define their composition. We can use the universal property of localization to show that the construction of $\Rig_R^*$ is well behaved with respect to $\frakm$-adic completion and base change:
		
		\begin{proposition}\label{p:rigcomp}
			There is a unique $\frakm$-adic completion functor $Rig_R^\dagger \to \Rig_R$ compatible with the $\frakm$-adic completion of weak formal schemes, in the sense that the following diagram commutes:
			\begin{equation*}
				\begin{tikzcd}
					\FS_R^\dagger	\arrow[d]	\arrow[r]	&	\FS_R	\arrow[d]	\\
					\Rig_R^\dagger	\arrow[r]	&	\Rig_R
				\end{tikzcd}
			\end{equation*}
		\end{proposition}
		\begin{proof}
			The $\frakm$-adic completion of an admissible weak blow-up is an admissible formal blow-up. It follows that the composition $\FS_R^\dagger \to \FS_R \to \Rig_R$ sends admissible weak blow-ups to isomorphisms. The universal property of localization gives the desired functor $\Rig_R^\dagger \to \Rig_R$.
		\end{proof}
		
		\begin{proposition}\label{p:rigbchange}
			Let $h:R \to S$ be a continuous map of complete Noetherian local rings. There is a unique \emph{weak base change} functor $\Rig_R^* \to \Rig_S^*$ compatible with (weak) formal base change in the sense that the following diagram commutes:
			\begin{equation*}
				\begin{tikzcd}
					\FS_R^*	\arrow[d]	\arrow[r]	&	\FS_S^*	\arrow[d]	\\
					\Rig_R^*	\arrow[r]	&	\Rig_S^*
				\end{tikzcd}
			\end{equation*}
		\end{proposition}
		\begin{proof}
			We need only show that the composition $\FS_R^* \to \FS_S^* \to \Rig_S^*$ sends admissible blow-ups to isomorphisms. Let $\calI$ be an open ideal of $\calO_\frakX$. Following the proof of Lemma \ref{l:bchange}, there is a diagram of ringed spaces
			\begin{equation*}
				\begin{tikzcd}
					\frakX_h'	\arrow[bend right,ddr,"\varphi"]	\arrow[bend left,drr]	\arrow[dr,"\psi"]	&	&	\\
						&	\frakY	\arrow[d,"\pi"]	\arrow[r]	&	\frakX_\calI	\arrow[d]	\\
						&	\frakX_h	\arrow[r,"f"]	&	\frakX
				\end{tikzcd}
			\end{equation*}
			where $\frakY$ is the completion of $\frakX_h \times_\frakX \frakX_\calI$ and $\varphi$ and $\psi$ are admissible blow-ups in $\FS_S^*$. It follows that $\pi^\rig$ is an isomorphism in $\Rig_S^*$.
		\end{proof}
		
\section{Weak Analytic Geometry}
		
	\subsection{Weak Rigid Analytic Spaces}
	
		Having constructed the category $\Rig_R^*$, our next goal is to attach to each object $\calX$ a locally ringed topos $(\calX_\ad,\calO_\calX)$ which we will refer to as a \emph{weak rigid analytic space} over $R$. In this section we construct the underlying topos $\calX_\ad$ as a ``limit'' of Zariski topoi along the admissible blow-ups $\frakX' \to \frakX$ of some model $\frakX$ of $\calX$. We defer the construction of the structure sheaf to \ref{s:an}. As is typical when dealing with ringed topoi, we will denote a geometric morphism by $f=(f^{-1},f_*)$, and reserve the notation $f^*$ for the pullback of modules.
		
		We say that a morphism $\calU \to \calX$ in $\Rig_R^*$ is an \emph{open immersion} if it admits a model $\frakU \to \frakX$ which is an open immersion in $\FS_R^*$. It is straightforward to see that a map $\frakU \to \frakX$ in $\FS_R^*$ becomes an open immersion in $\Rig_R^*$ if and only if there is a commutative diagram
		\begin{equation*}
			\begin{tikzcd}
				\frakU'	\arrow[d]	\arrow[r,"f"]	&	\frakX'	\arrow[d]	\\
				\frakU	\arrow[r]	&	\frakX
			\end{tikzcd}
		\end{equation*}
		where $f$ is an open immersion and the vertical arrows are admissible blow-ups. Consequently, for every finite family $\{\calX_i \to \calX\}_i$ of open immersions in $\Rig_R^*$, there is a model $\frakX$ of $\calX$ and models $\frakX_i \to \frakX$ of $\calX_i \to \calX$ such that $\{\frakX_i \to \frakX\}_i$ is a family of Zariski open immersions.
		
		For an object $\calX$ of $\Rig_R^*$, define the category $\Ad_\calX$ to be the category of all open immersions $\calU \to \calX$. If $\frakX$ is an object of $\FS_R^*$, we let $\Zar_\frakX$ denote the Zariski site of $\frakX$. For clarity of notation, we will identify $\frakX$ with its Zariski topos. The \emph{admissible topology} on $\Ad_\calX$ is defined to be the coarsest topology such that for every model $\frakX$ of $\calX$, the canonical map $\Zar_\frakX \to \Ad_\calX$ is continuous. We say that a family of open immersions $\{\calX_i \to \calX\}_i$ is an \emph{admissible covering} if it is a covering in the admissible topology on $\calX$. Let $\calX_\ad$ be the topos of sheaves on $\Ad_\calX$.
		
		\begin{proposition}\label{p:spe}
			The map $\Zar_\frakX \to \Ad_\calX$ preserves finite limits. In particular, there is a geometric morphism $\spe:\calX_\ad \to \frakX$ which we refer to as the \emph{specialization map}.
		\end{proposition}
		\begin{proof}
			This is immediate from Lemmas \ref{l:comp}-\ref{l:coeq}.
		\end{proof}
		
		\begin{lemma}\label{l:open}
			The base change of an open immersion in $\Rig_R^*$ is an open immersion.
		\end{lemma}
		\begin{proof}
			Let $\calV \to \calY$ be an open immersion in $\Rig_R^*$, and choose a model $\frakV \to \frakY$ in $\FS_R^*$. Let $f:\calX \to \calY$ be any morphism and $\frakX \to \frakZ$ a model. Since $\frakY^\rig \cong \frakZ^\rig$, there is a diagram $\frakZ \leftarrow \frakY' \to \frakY$, where both arrows are admissible blow-ups. Consider the diagram
			\begin{equation}\label{d:open}
				\begin{tikzcd}
					\frakU	\arrow[d]	\arrow[r]	&	\frakY' \times_\frakY \frakV	\arrow[d]	\arrow[r]	&	\frakV	\arrow[d]	\\
					\frakX \times_\frakZ \frakY'	\arrow[d]	\arrow[r]	&	\frakY'	\arrow[d]	\arrow[r]	&	\frakY	\\
					\frakX	\arrow[r]	&	\frakZ	&
				\end{tikzcd}
			\end{equation}
			where the top left square is cartesian. Then $(\frakU \to \frakX)^\rig$ is an open immersion, and by Proposition \ref{p:spe} this map agrees with the base change of $\calV \to \calY$ along $\calX \to \calY$.
		\end{proof}
		
		If $\{\frakU_i \to \frakX\}_i$ is a Zariski covering of some model $\frakX$ of $\calX$, then by Proposition \ref{p:spe} $\{(\frakU_i \to \frakX)^\rig\}_i$ is a covering in $\Ad_\calX$. Using Lemma \ref{l:open}, we can verify that the collection of all such coverings on $\Ad_\calX$ constitutes a Grothendieck pretopology. The corresponding topology is necessarily coarser than the admissible topology, however, the functors $\Zar_\frakX \to \Ad_\calX$ are clearly continuous with respect to this topology as well. It follows that this topology coincides with the admissible topology on $\Ad_\calX$.
		
		\begin{proposition}\label{p:admor}
			Let $f:\calX \to \calY$ be a morphism in $\Rig_R^*$. The base change $\Ad_\calY \to \Ad_\calX$ is a morphism of sites, and therefore induces a geometric morphism $f:\calX_\ad \to \calY_\ad$.
		\end{proposition}
		\begin{proof}
			Clearly $\calY = \calY \times_\calX \calX$, so $\Ad_\calY \to \Ad_\calX$ preserves terminal objects. Moreover, this functor preserves fiber products by Proposition \ref{p:spe}. We need only show that the base change along $f$ of an admissible covering of $\calY$ is an admissible covering of $\calX$. By the preceding discussion, it suffices to show the statement for a covering of the form $\{(\frakV_i \to \frakY)^\rig\}_i$, where $\frakY$ is a model of $\calY$ and $\{\frakV_i \to \frakY\}_i$ is a Zariski covering. We assign to each $\frakV_i \to \frakY$ an open immersion $\frakU_i \to \frakX \times_\frakZ \frakY'$ as in the proof of Lemma \ref{l:open}. These maps constitute a Zariski covering of $\frakX \times_\frakZ \frakY'$, and their image in $\Rig_R^*$ is an admissible covering of $(\frakX \times_\frakZ \frakY')^\rig \cong \calX$.
		\end{proof}
		
		Fix a model $\frakX$ of $\calX$. It will be convenient to assemble all of the admissible blow-ups of $\frakX$ and their Zariski sites into a single object. Define the category $\Tot_\frakX$ as follows: the objects of $\Tot_\frakX$ are pairs $(\varphi,\frakU)$, where $\varphi:\frakX_\varphi \to \frakX$ is an admissible blow-up and $\frakU$ is an open (weak) formal subscheme of $\frakX_\varphi$. A morphism $(\varphi,\frakU) \to (\psi,\frakV)$ is an $\frakX$-morphism $f:\frakX_\varphi \to \frakX_\psi$ satisfying $f(\frakU) \subseteq \frakV$. Let $\Bl_\frakX$ denote the category of admissible blow-ups $\varphi:\frakX_\varphi \to \frakX$. There is a forgetful functor
		\begin{equation*}
			\Tot_\frakX	\to	\Bl_\frakX.
		\end{equation*}
		We may equip $\Tot_\frakX$ with the structure of a fibered category over $\Bl_\frakX$ by declaring that a morphism $f:(\varphi,\frakU) \to (\psi,\frakV)$ is cartesian whenever $f^{-1}(\frakV) = \frakU$. The fiber over $\varphi:\frakX_\varphi \to \frakX$ is canonically isomorphic to the Zariski site $\Zar_{\frakX_\varphi}$. We denote by
		\begin{equation*}
			\alpha_{\varphi !}:\Zar_{\frakX_\varphi} \to \Tot_\frakX
		\end{equation*}
		the inclusion of the fiber. The \emph{total topology} on $\Tot_\frakX$ is the coarsest topology for which the maps $\alpha_{\varphi !}$ are continuous. The total topology is also the finest topology for which the maps $\alpha_{\varphi !}$ are cocontinuous (\cite{SGA4}, VI. 7.4.3). The \emph{total topos} $\frakX_\tot$ of $\frakX$ is defined to be the topos of sheaves on $\Tot_\frakX$.
		
		From the construction of the category $\Rig_R^*$, the functor
		\begin{equation}\label{eq:pi}
			\Tot_\frakX	\to	\Ad_\calX
		\end{equation}
		sending $(\varphi,\frakU)$ to the open immersion $(\frakU \to \frakX)^\rig$ is precisely the localization of $\Tot_\frakX$ at the class of cartesian morphisms. In the terminology of (\cite{SGA4}, VI. 6.3), $\Ad_\calX$ is the \emph{inductive limit} of $\Tot_\frakX$ over $\Bl_\frakX^\op$. The terminology means the following: the choice of cleavage of the fibration $\Tot_\frakX \to \Bl_\frakX$ uniquely determines a pseudo-functor
		\begin{equation}\label{eq:pseudo}
			\Bl_\frakX^\op	\to	\Cat
		\end{equation}
		sending $\varphi:\frakX_\varphi \to \frakX$ to $\Zar_{\frakX_\varphi}$ and $f:\frakX_\varphi \to \frakX_\psi$ to the functor $f^{-1}:\Zar_{\frakX_\psi} \to \Zar_{\frakX_\varphi}$. The category $\Ad_\calX$ is the inductive limit of this pseudo-functor.
		
		The functor (\ref{eq:pi}) is evidently a morphism of sites, and therefore induces a geometric morphism
		\begin{equation*}
			\pi:\calX_\ad	\to	\frakX_\tot
		\end{equation*}
		As the functors $\alpha_{\varphi !}$ are both continuous and cocontinuous, they give rise to an adjoint triple (\cite{SGA4}, VI. 7.4.3):
		\begin{align*}
			\alpha_{\varphi !}:\frakX_\varphi &\to	\frakX_\tot	&	\alpha_\varphi^{-1}:\frakX_\tot	&\to	\frakX_\varphi	&	\alpha_{\varphi *}:\frakX_\varphi	&\to	\frakX_\tot.
		\end{align*}
		Let
		\begin{align*}
			a_\varphi = (\alpha_{\varphi !},\alpha_\varphi^{-1}):	\frakX_\tot	&\to	\frakX_\varphi	\\
			\alpha_\varphi = (\alpha_\varphi^{-1},\alpha_{\varphi *}):	\frakX_\varphi	&\to	\frakX_\tot
		\end{align*}
		be the corresponding morphisms of topoi, and $\mu_\varphi = a_\varphi \circ \pi:\calX_\ad \to \frakX_\varphi$. Note that $\mu_\varphi$ coincides with the specialization map $\spe:\calX_\ad \to \frakX_\varphi$ of Proposition \ref{p:spe}.
		
		\begin{proposition}\label{p:shad}
			The following data are equivalent:
			\begin{enumerate}
				\item	A sheaf $\calF$ in $\frakX_\tot$
				\item	For each admissible blow-up $\varphi:\frakX_\varphi \to \frakX$, a sheaf $\calF_\varphi$ on $\frakX_\varphi$, and for each morphism $f:\frakX_\varphi \to \frakX_\psi$ in $\Bl_\frakX$ a morphism $\gamma_f(\calF):\calF_{\varphi_2} \to f_* \calF_{\varphi_1}$, satisfying the cocycle condition
					\begin{equation}\label{eq:cocycle}
						\gamma_{g \circ f}(\calF)	=	z_{g,f}	\circ	g_*(\gamma_f(\calF)) \circ \gamma_g(\calF).
					\end{equation}
			\end{enumerate}
		\end{proposition}
		\begin{proof}
			This is (\cite{SGA4}, VI. 7.4.7). We briefly recall the construction for (i) $\Rightarrow$ (ii). Let $\calF$ be an object of $\frakX_\tot$. For each admissible blow-up $\varphi:\frakX_\varphi \to \frakX$, we have $\calF_\varphi = \alpha_\varphi^{-1} \calF$. Let $f:\frakX_\varphi \to \frakX_\psi$ be a morphism in $\Bl_\frakX$. For each open $\frakU \subseteq \frakX_2$, we have a unique cartesian morphism covering $f$
			\begin{equation*}
				\alpha_{\varphi_1 !} f^{-1} \frakU	\to	\alpha_{\varphi_2 !} \frakU.
			\end{equation*}
			Applying $\calF$ to this morphism, we get a map
			\begin{equation*}
				\calF_{\varphi_2}	=	\calF(\alpha_{\varphi_2 !} \frakU)	\to	\calF(\alpha_{\varphi_1 !} f^{-1} \frakU)	=	f_* \calF_{\varphi_1}.
			\end{equation*}
		\end{proof}
		
		Note in particular that the maps $\gamma_f(\calF)$ of Proposition \ref{p:shad} depend only on the choice of cleavage of $\Tot_\frakX \to \Bl_\frakX$. Consequently, a morphism $\alpha:\calF \to \calG$ in $\frakX_\tot$ is an isomorphism if and only if $\alpha_\varphi = \mu_\varphi^{-1} \alpha$ is an isomorphism for all $\varphi$. The characterization of $\Ad_\calX$ as the inductive limit of this fibration gives us a useful description of the functor $\pi^{-1}$ (\cite{SGA4}, 8.5.2):
		\begin{equation}\label{eq:pistar}
			\pi^{-1}\calF	=	\varinjlim_{\varphi \in \Bl_\frakX^\op}	\mu_\varphi^{-1} \calF_\varphi.
		\end{equation}
		Here, given a morphism $f:\frakX_\varphi \to \frakX_\psi$ in $\Bl_\frakX$, the transition map is given by the composition
		\begin{equation*}
			\mu_\psi^{-1}	\calF_\psi	\xrightarrow{\sim}	\mu_\varphi^{-1} f^{-1} \calF_\psi	\to	\mu_\varphi^{-1} \calF_\varphi,
		\end{equation*}
		where the second map $\mu_\varphi^{-1}$ applied to the adjunction of $\gamma_f(\calF)$. Finally, we have the following description of the topos $\calX_\ad$ as a subtopos of the total topos $\frakX_\tot$:
		
		\begin{theorem}\label{t:shad}
			The functor $\pi_*$ is fully faithful. Its essential image is the subcategory of sheaves $\calF$ in $\frakX_\tot$ for which the maps $\gamma_f(\calF)$ are isomorphisms for all morphisms $f$ in $\Bl_\frakX$.
		\end{theorem}
		\begin{proof}
			This is (\cite{SGA4}, VI. 8.2.9-8.2.10).
		\end{proof}
		
		Theorem \ref{t:shad} gives us a method of constructing sheaves on $\Ad_\calX$ by assembling compatible collections of sheaves on the admissible blow-ups $\frakX' \to \frakX$. In the following section, we use this method to equip each $\calX_\ad$ with the structure of a ringed topos. To conclude this section, we give some remarks on the compatibility of the admissible topology with $\frakm$-adic completion and base change:
		
		\begin{proposition}\label{p:adcomp}
			Let $\calX$ be an object of $\Rig_R^\dagger$. Then $\frakm$-adic completion induces a morphism of sites $\Ad_\calX \to \Ad_{\calX_\infty}$ and consequently a geometric morphism $(\calX_\infty)_\ad \to \calX_\ad$.
		\end{proposition}
		\begin{proof}
			Note that for any model $\frakX$ of $\calX$, the composition
			\begin{equation}\label{eq:comp}
				\Zar_\frakX \to \Zar_{\frakX_\infty}	\to	\Ad_{\calX_\infty}
			\end{equation}
			is a morphism of sites. Given a Zariski covering $\{\frakU_i \to \frakX\}_i$, its image under (\ref{eq:comp}) is an admissible covering of $\calX_\infty$. Since the admissible topology is generated by coverings of the form $\{(\frakU_i \to \frakX)^\rig\}_i$, we see that $\frakm$-adic completion sends admissible coverings to admissible coverings. It remains to show that this functor sends fiber products to fiber products. But given two open immersions $\calU_1 \to \calX$ and $\calU_2 \to \calX$, there exists a model $\frakX$ of $\calX$ and open immersions $\frakU_i \to \frakX$ such that $(\frakU_i \to \frakX)^\rig = \calU_i \to \calX$. The statement follows from the fact that both $\Zar_\frakX \to \Ad_\calX$ and (\ref{eq:comp}) preserves fiber products.
		\end{proof}
		
		\begin{proposition}\label{p:adbchange}
			Let $h:R \to S$ be a continuous map of complete Noetherian local rings. Then for any object $\calX$ of $\Rig_R^*$, the natural map $\Ad_\calX \to \Ad_{\calX_h}$ is a morphism of sites, and therefore induces a geometric morphism $(\calX_h)_\ad \to \calX_\ad$.
		\end{proposition}
		\begin{proof}
			For any model $\frakX$ of $\calX$, there is a morphism of sites
			\begin{equation*}
				\Zar_\frakX	\to	\Zar_{\frakX_h}	\to	\Ad_{\calX_h}
			\end{equation*}
			The result follows by arguing as in the proof of Proposition \ref{p:adcomp}.
		\end{proof}
		
	\subsection{Analytification}\label{s:an}
	
		Let $\frakX$ be an object of $\FS_R^*$, and $\calX = \frakX^\rig$. Our next goal is to associate to each $\calO_\frakX$-module $\calF$ a sheaf $\calF^\rig$ on $\Ad_\calX$, called the \emph{analytification} of $\calF$. Letting $\calO_\calX = \calO_\frakX^\rig$, we equip the topos $\calX_\ad$ with the structure of a (locally) ringed topos. The basic procedure is given by Theorem \ref{t:shad}: we construct for each such $\calF$ a sheaf $\hat{\calF}$ on the total site $\Tot_\frakX$ such that the morphisms $\gamma_f(\hat{\calF})$ of Proposition \ref{p:shad} are isomorphisms. In particular, the structure sheaf $\calO_\calX$ will be independent of the choice of model $\frakX$.
		
		\begin{definition}
			Let $\calF$ be an $\calO_\frakX$-module. The \emph{$\frakm$-closure} of $\calF$ is defined to be the sheaf
			\begin{equation*}
				\calH^0(\calF)	=	\varinjlim_n	\calH om_{\calO_\frakX}(\frakm^n \calO_\frakX, \calF).
			\end{equation*}
		\end{definition}
		
		The assignment $\calF \to \calH^0(\calF)$ is evidently functorial. This functor is compatible with restriction in the sense that $\calH^0(\calF|_\frakU) = \calH^0(\calF)|_\frakU$. It follows that for every point $x$ of $\frakX$,
		\begin{equation}\label{eq:h0stalk}
			\calH^0(\calF)_x	=	\varinjlim_n	\Hom_{\calO_{\frakX,x}}(\frakm^n \calO_{\frakX,x}, \calF_x).
		\end{equation}
		There is a natural pairing
		\begin{equation*}
			\calH^0(\calO_\frakX) \times \calH^0(\calF) \to \calH^0(\calF)
		\end{equation*}
		which sends $f:\frakm^m \calO_\frakX \to \calO_\frakX$ and $g:\frakm^n \calO_\frakX \to \calF$ to the composition
		\begin{equation*}
			\frakm^{m+n} \calO_\frakX	\xrightarrow{f}	\frakm^n \calO_\frakX	\xrightarrow{g}	\calF.
		\end{equation*}
		In particular, we may regard $\calH^0(\calO_\frakX)$ as a sheaf of $\calO_\frakX$-algebras, and $\calH^0$ as a functor
		\begin{equation*}
			\Mod(\calO_\frakX) \to \Mod(\calH^0(\calO_\frakX)).
		\end{equation*}
		
		To elucidate the meaning of this construction, consider the case when $\frakm \calO_\frakX = \pi \calO_\frakX$ is generated by a single global section. Suppose moreover that $\frakX$ is $\frakm$-torsion free. Then a section in $\calH om_{\calO_\frakX}(\frakm^n \calO_\frakX, \calO_\frakX)$ may be regarded as a section of the $\calO_\frakX$-module $\pi^{-n} \calO_\frakX$. Consequently we have
		\begin{equation*}
			\calH^0(\calO_\frakX)	=	\calO_\frakX[\pi^{-1}].
		\end{equation*}
		For the purposes of rigid geometry, the assumptions on $\frakX$ are not too restrictive: replacing $\frakX$ by the admissible blow-up $\frakX_\frakm$, we can usually assume that $\frakm \calO_\frakX$ is locally principal and $\frakX$ is $\frakm$-torsion free.
		
		Let $\Coh(\calO_\frakX)$ denote the full subcategory of $\Mod(\calO_\frakX)$ consisting of coherent $\calO_\frakX$-modules. The restriction of $\calH^0$ to $\Coh(\calO_\frakX)$ enjoys many good properties:
		
		\begin{lemma}\label{l:h0comp}
			Let $\frakX$ be an object of $\FS_R^\dagger$, and $\calF$ a coherent $\calO_\frakX$-module. Then $\calH^0(\calF)_\infty = \calH^0(\calF_\infty)$.
		\end{lemma}
		\begin{proof}
			Recall that we have a morphism of ringed spaces $\iota:\frakX_\infty \to \frakX$. For each $n$, $\frakm$-adic completion induces a morphism
			\begin{equation*}
				\calH om_{\calO_\frakX}(\frakm^n \calO_\frakX, \calF)	\to	\calH om_{\calO_{\frakX_\infty}}(\frakm^n \calO_{\frakX_\infty},\calF_\infty).
			\end{equation*}
			Note that $\calH om_{\calO_\frakX}(\frakm^n \calO_\frakX, \calF)$ is coherent, so passing to $\frakm$-adic completions gives a morphism
			\begin{equation*}
				\iota^*\calH om_{\calO_\frakX}(\frakm^n \calO_\frakX, \calF)	\to	\calH om_{\calO_{\frakX_\infty}}(\frakm^n \calO_{\frakX_\infty},\calF_\infty).
			\end{equation*}
			We claim that this is an isomorphism. It suffices to check that the induced map on stalks
			\begin{equation*}
				\Hom_{\calO_{\frakX,x}}(\frakm^n \calO_{\frakX,x},\calF_x) \otimes_{\calO_{\frakX,x}} \calO_{\frakX_\infty,x}	\to	\Hom_{\calO_{\frakX_\infty,x}}(\frakm^n \calO_{\frakX_\infty,x},(\calF_\infty)_x)
			\end{equation*}
			is an isomorphism. But this follows easily from the fact that $\calO_\frakX \to \calO_{\frakX_\infty}$ is flat and $\frakm^n \calO_\frakX$ is finitely presented. Passing to direct limits, we obtain the desired isomorphism $\calH^0(\calF)_\infty \to \calH^0(\calF_\infty)$.
		\end{proof}
		
		\begin{proposition}
			Suppose that $\frakm \calO_\frakX$ is locally principal. Then $\calH^0$ restricts to an exact functor
			\begin{equation*}
				\calH^0:\Coh(\calO_\frakX) \to \Coh(\calH^0(\calO_\frakX)).
			\end{equation*}
			If $\calF$ is a coherent $\calO_\frakX$-module, then the natural map $\calF \to \calH^0(\calF)$ induces an isomorphism
			\begin{equation*}
				\calF	\otimes_{\calO_\frakX}	\calH^0(\calO_\frakX)	\to	\calH^0(\calF).
			\end{equation*}
		\end{proposition}
		\begin{proof}
			Recall that we have a morphism of ringed spaces $\iota:\frakX_\infty \to \frakX$. Consider an exact sequence $0 \to	\calF' \to \calF \to \calF'' \to 0$ in $\Coh(\calO_\frakX)$. Applying $\iota^* \calH^0=\calH^0 \iota^*$, we obtain a diagram
			\begin{equation*}
				0	\to	\calH^0(\calF'_\infty)	\to	\calH^0(\calF)	\to	\calH^0(\calF''_\infty)	\to	0.
			\end{equation*}
			By (\cite{Abbes}, 2.10.18), this sequence is exact. Since $\iota^*$ is faithfully exact, we see that
			\begin{equation*}
				0	\to	\calH^0(\calF')	\to	\calH^0(\calF)	\to	\calH^0(\calF'')	\to	0
			\end{equation*}
			is an exact sequence of $\calO_\frakX$-modules. We claim that $\calH^0(\calF)$ is a coherent $\calH^0(\calO_\frakX)$-module. This is a local question, so assume we have a presentation
			\begin{equation}\label{eq:pres}
				\calO_\frakX^m	\to	\calO_\frakX^n	\to	\calF	\to	0.
			\end{equation}
			The result then follows from exactness of $\calH^0$. The final statement can be seen by tensoring the sequence (\ref{eq:pres}) with $\calH^0(\calO_\frakX)$.
		\end{proof}
		
		Let $\calF$ be a coherent $\calO_\frakX$-module. We will now construct the associated sheaf $\hat{\calF}$ on $\Tot_\frakX$. For each morphism $f:\frakY \to \frakX$ in $\FS_R^*$, we have a map
		\begin{equation*}
			\beta_f(\calF):	\calH^0(\calF)	\to	f_* \calH^0(f^* \calF)),
		\end{equation*}
		defined by composition of the canonical maps
		\begin{equation*}
			\varinjlim_n	\calH om_{\calO_\frakX}(\frakm^n\calO_\frakX,\calF)	\to	\varinjlim_n	f_*\calH om_{\calO_\frakX}(f^*\frakm^n\calO_\frakX,f^*\calF)	\to	f_*	\varinjlim_n	\calH om_{\calO_\frakX}(\frakm^n\calO_\frakY,\calF).
		\end{equation*}
		By Lemma \ref{l:h0comp} and the fact that $\iota^*$ commutes with direct limits, the $\frakm$-adic completion of $\beta_f(\calF)$ is the map $\beta_{f_\infty}(\calF_\infty)$. Since $\iota^*$ is faithfully exact, by (\cite{Abbes}, 2.10.28) these maps satisfy a cocycle condition expressed by the commutativity of the following diagram:
		\begin{equation}\label{eq:cocycle2}
			\begin{tikzcd}
				\calH^0(\calF)	\arrow[d,"\beta_f(\calF)"']	\arrow[r,"\beta_{f \circ g}(\calF)"]	&	(f \circ g)_* \calH^0((f \circ g)^* \calF)	\arrow[d,"\sim"]	\\
				f_* \calH^0(f^* \calF)	\arrow[r,"f_* \beta_g(f^* \calF)"]	&	f_* g_* \calH^0(g^* f^* \calF)
			\end{tikzcd}
		\end{equation}
		For each admissible blow-up $\varphi:\frakX_\varphi \to \frakX$, we define $\hat{\calF}_\varphi$ to be the sheaf $\calH^0(\varphi^* \calF)$. For each morphism $f:\frakX_\varphi \to \frakX_\psi$, we define
		\begin{equation*}
			\gamma_f(\hat{\calF}):	\calH^0(\psi^* \calF)	\xrightarrow{\beta_f(\psi^* \calF)}	f_* \calH^0(f^* \psi^* \calF)	\xrightarrow{\sim}	f_* \calH^0(\varphi^* \calF).
		\end{equation*}
		The diagram (\ref{eq:cocycle2}) indicates that these maps satisfy the cocycle condition (\ref{eq:cocycle}) and therefore determine a unique sheaf $\hat{\calF}$ on $\Tot_\frakX$. To see that this sheaf descends to a sheaf on $\Ad_\calX$, we require the following:
		
		\begin{theorem}\label{t:acyc}
			(Tate's Acyclicity Theorem). For every coherent sheaf $\calF$ of $\calO_\frakX$-modules, the maps $\gamma_f(\hat{\calF})$ are isomorphisms. Consequently, $\hat{\calF}$ descends to a sheaf $\calF^\rig$ on $\Ad_\calX$.
		\end{theorem}
		\begin{proof}
			Let $f:\frakX_\varphi \to \frakX_\psi$ be a morphism in $\Bl_\frakX$. We must show that the map
			\begin{equation}\label{eq:betaf}
				\beta_f(\psi^* \calF):\calH^0(\psi^* \calF)	\to	f_* \calH^0(f^* \psi^* \calF)
			\end{equation}
			is an isomorphism. But we may check this on the level of $\frakm$-adic completions, and the result follows from (\cite{Abbes}, 3.5.5).
		\end{proof}
		
		In light of the theorem, we may define the sheaf $\calF^\rig$ to be the sheaf on $\Ad_\calX$ associated to $\hat{\calF}$, that is, $\calF^\rig = \pi^{-1}\calF$. Recall that for every admissible blow-up $\varphi:\frakX_\varphi \to \frakX$, we have a geometric morphism $\mu_\varphi:\calX_\ad \to \frakX_\varphi$. By (\ref{eq:pistar}), we see that
		\begin{equation}\label{eq:indlim}
			\calF^\rig	\cong	\varinjlim_{\varphi \in \Bl_\frakX^\op}	\mu_\varphi^{-1} \calH^0(\varphi^* \calF).
		\end{equation}
		
		We define $\calO_\calX = \calO_\frakX^\rig$ for some model $\frakX$ of $\calX$. Note that this definition does not depend on the choice of model. Together, the pair $(\calX_\ad,\calO_\calX)$ constitutes a ringed topos. We will say that a $\calO_\calX$-module is \emph{coherent} if it is of the form $\calF^\rig$, where $\calF$ is a coherent $\calO_\frakX$-module for some model $\frakX$ of $\calX$. As in (\cite{Abbes}, 4.8.18), one may show that this is equivalent to the usual definition of a coherent module, but we will not need this result. Our next goal is to explicitly describe the stalks of coherent $\calO_\calX$-modules. Recall that a \emph{point} of the topos $\calX_\ad$ is a geometric morphism $P:\Set \to \calX_\ad$. For any $\calF \in \Ad_\calX$, the \emph{stalk} of $\calF$ at $P$ is the set $P^{-1} \calG$. The composition
		\begin{equation*}
			\Set	\xrightarrow{P}	\calX_\ad	\xrightarrow{\spe}	\frakX
		\end{equation*}
		corresponds to a Zariski point of $\frakX$ which we refer to as the \emph{specialization} of $P$ in $\frakX$.
		
		\begin{theorem}
			The pair $(\calX_\ad,\calO_\calX)$ is a locally ringed topos.
		\end{theorem}
		\begin{proof}
			Let $\frakX$ be a model of $\calX$. As the question is local in nature, we may assume that $\frakm \calO_\frakX = \pi \calO_\frakX$ is principal. If $P:\Set \to \calX_\ad$ is a point of $\calX_\ad$, then for any admissible blow-up $\varphi:\frakX_\varphi \to \frakX$, let $x(\varphi)=\mu_\varphi \circ P$ denote the specialization of $P$ in $\frakX_\varphi$. By (\ref{eq:indlim}), we see that
			\begin{equation*}
				P^{-1} \calO_\calX	=	\varinjlim_{\varphi \in \Bl_\frakX^\op}	\calH^0(\calO_{\frakX_\varphi})_{x(\varphi)}	=	\varinjlim_{\varphi \in \Bl_\frakX^\op} \left(	\calO_{\frakX_\varphi,x(\varphi)}[\pi^{-1}]	\right)	=	A[\pi^{-1}],
			\end{equation*}
			where $A=\varinjlim_\varphi \calO_{\frakX_\varphi,x(\varphi)}$. Then $A$ is a local ring, and $\pi$ is contained in the maximal ideal of $A$. By (\cite{Abbes}, 1.9.4), to show that $A[\pi^{-1}]$ is local it suffices to show that every finitely generated open ideal $I$ of $A$ is invertible. From the definition of $A$, there necessarily exists an admissible blow-up $\varphi:\frakX_\varphi \to \frakX$, an open subset $\frakU \subseteq \frakX_\varphi$, and an open ideal $\calI$ on $\frakU$ whose image in $A$ is $I$. Without loss of generality, we may assume that $\frakU=\frakX_\varphi$. By blowing up $\frakX$ at $\frakm \calO_\frakX$, we may also assume that $\frakX$ is $\frakm$-torsion free. But if $\psi:(\frakX_\varphi)_\calI \to \frakX'$ is the admissible blow-up of $\frakX_\varphi$ at $\calI$, then $\psi^{-1} \calI$ is principal in a neighborhood of $x(\varphi \circ \psi)$. In particular, $I$ is a principal open ideal in $A$. Since $\pi A$ is invertible, it follows that $I$ is invertible as well.
		\end{proof}
		
		\begin{remark}
			If $\frakX$ is a model of $\calX$, and $\calF$ is a coherent $\calO_\frakX$-module, the in the notation of the proof we have
			\begin{equation}\label{eq:rigstalk}
				P^{-1} \calF^\rig	=	\calF_x	\otimes_{\calO_{\frakX,x}} A[\pi^{-1}].
			\end{equation}
		\end{remark}
		
		There is a morphism of ringed topoi
		\begin{equation*}
			(\calX_\ad,\calO_\calX)	\to	(\frakX_\zar,\calO_\frakX)
		\end{equation*}
		given by the specialization map $\spe:\calX_\ad \to \frakX_\zar$ and the canonical map $\calO_\frakX \to \calH^0(\calO_\frakX) = \spe_* \calO_\calX$. To conclude the section, we verify that our basic constructions for weak rigid analytic spaces correspond to morphisms of ringed topoi:
		
		\begin{proposition}
			Let $f:\calX \to \calY$ be a morphism in $\Rig_R^*$. Then $f$ induces a morphism of ringed topoi $(\calX_\ad,\calO_\calX) \to (\calY_\ad,\calO_\calY)$.
		\end{proposition}
		\begin{proof}
			The underlying map of topoi is given by Proposition \ref{p:admor}. We will construct a morphism $\calO_\calY \to f_* \calO_\calX$. Let $\calV \to \calY$ be an open immersion. Choose formal models $\frakX \to \frakZ$ and $\frakV \to \frakY$ as in the proof of Lemma \ref{l:open}. Referring to the diagram (\ref{d:open}), let $\calU = \frakU^\rig$, so that $\calU \to \calX$ is an open immersion. Then we have
			\begin{equation*}
				(\calO_\calX)_*(\calV)	=	\calO_\calX(\calU)	=	\calH^0(\calO_{\frakX \times_\frakZ \frakY'})(\frakU).
			\end{equation*}
			Similarly,
			\begin{equation*}
				\calO_\calY(\calV)	=	\calH^0(\calO_{\frakY'})(\frakY' \times_\frakY \frakV).
			\end{equation*}
			The map $\frakX \times_\frakX \frakY' \to \frakY'$ induces a map $\calO_{\frakY'}(\frakY' \times_\frakY \frakV) \to \calO_{\frakX \times_\frakZ \frakY'}(\frakU)$, and the desired map $\calO_\calY(\calV) \to (\calO_\calX)_*(\calV)$ is given by functoriality of $\calH^0$.
		\end{proof}
		
		\begin{proposition}
			Let $\calX$ be an object of $\FS_R^\dagger$. Then $\frakm$-adic completion induces a morphism of ringed topoi $((\calX_\infty)_\ad, \calO_{\calX_\infty}) \to (\calX_\ad,\calO_\calX)$.
		\end{proposition}
		\begin{proof}
			The underling map of topoi $\iota$ is given by Proposition \ref{p:adcomp}. Let $\frakU \to \frakX$ be a model of an open immersion $\calU \to \calX$. Then 
			\begin{equation*}
				\iota_*\calO_{\calX_\infty}(\calU)	=	\calO_{\calX_\infty}(\calU_\infty)	=	\calH^0(\calO_{\frakX_\infty})(\frakU_\infty).
			\end{equation*}
			By definition, $\calO_\calX(\calU)	=	\calH^0(\calO_\frakX)(\frakU)$. The desired map is given by the $\frakm$-adic completion $\calO_\frakX(\frakU) \to \calO_{\frakX_\infty}(\frakU_\infty)$ and functoriality of $\calH^0$.
		\end{proof}
		
		\begin{proposition}
			Let $h:R \to S$ be a continuous map of Noetherian local domains. Then base change induces a morphism of ringed topoi $((\calX_h)_\ad,\calO_{\calX_h}) \to (\calX_\ad,\calO_\calX)$.
		\end{proposition}
		\begin{proof}
			The underlying morphism of topoi $h$ is given by Proposition \ref{p:adbchange}. Again we construct a morphism $\calO_\calX \to h_* \calO_{\calX_h}$. Let $\frakU \to \frakX$ be a model of $\calU \to \calX$. Then
			\begin{equation*}
				h_* \calO_{\calX_h}(\calU)	=	\calH^0(\calO_{\frakX_h}) ((\frakU \times_\frakX \frakX_h)_\infty).
			\end{equation*}
			By definition, $\calO_\calX(\calU)	=	\calH^0(\calO_\frakX)(\frakU)$. Now the map of ringed spaces $\frakX_h \to \frakX$ gives a morphism $\calO_\frakX(\frakU) \to \calO_{\frakX_h} ((\frakU \times_\frakX \frakX_h)_\infty)$, and the desired map is given by functoriality of $\calH^0$.
		\end{proof}
	
	\subsection{Rigid Points}\label{s:points}
		
		Let $\calX$ be an object of $\Rig_R^*$. As in (\cite{Abbes}, 4.4.6), one may show that the topos $\calX_\ad$ is a coherent topos, and therefore has \emph{enough points}, in the sense that isomorphisms in $\calX_\ad$ can be checked on stalks. In applications, however, the collection of all points of the topos $\calX_\ad$ is somewhat inconvenient, for example: there is no evident notion of the fiber of a morphism $\calX \to \calY$ at an arbitrary point of $\calY_\ad$. In this final section we construct an underlying set $\langle \calX \rangle$ of \emph{rigid points}, which may be interpreted as certain closed immersions in the category $\Rig_R$. Our main theorem states that isomorphisms in $\calX_\ad$ can be checked at the stalks of rigid points, thus making the set $\langle \calX \rangle$ a convenient object for doing geometry.
		
		\begin{definition}
			An $R$-algebra $\Omega$ is an \emph{$R$-point} if $\Omega$ is a $1$-dimensional local domain of finite type over $R$.
		\end{definition}
		
		\begin{lemma}\label{l:rpoints}
			Let $\Omega$ be an $R$-point. Then
			\begin{enumerate}
				\item	$\Omega$ is complete
				\item	$\Omega$ is a finite $R$-module
				\item	The integral closure $\overline{\Omega}$ in $Q(\Omega)$ is a discrete valuation ring.
			\end{enumerate}
		\end{lemma}
		\begin{proof}
			By (\cite{Abbes}, 1.11.4), the statements are true for the $\frakm$-adic completion of $\Omega$. But then $\Omega$ is a finite module over a complete Noetherian ring, and is therefore $\frakm$-adically complete.
		\end{proof}
		
		Note that the lemma implies that every $R$-point $\Omega$ is automatically a w.c.f.g. algebra over $R$. The locally ringed spaces $\Spwf(\Omega)$ and $\Spf(\Omega)$ are isomorphic, and we will generally not distinguish the two. If $\frakX$ is an object of $\FS_R^*$, we let $\langle \frakX \rangle$ denote the set of equivalence classes of immersions
		\begin{equation*}
			P:\Spwf(\Omega)	\to	\frakX,
		\end{equation*}
		where $\Omega$ is an $R$-point. We say that $\langle \frakX \rangle$ is the set of \emph{rigid points} of $\frakX$. When there is no risk of confusion, we may identify an immersion $P$ with its associated rigid point. The \emph{residue field} of $P$ is defined to be the field of fractions $k(P)=Q(\Omega)$. The field $k(P)$ is naturally equipped with a discrete valuation coming from the integral closure $\overline{\Omega}$.
		
		\begin{lemma}
			Every rigid point $P$ is a closed immersion.
		\end{lemma}
		\begin{proof}
			It suffices to show the statement when $\frakX=\Spwf(A)$ is affine. Factor $P$ as
			\begin{equation*}
				\Spwf(\Omega)	\to	\frakU	\to	\frakX,
			\end{equation*}
			where $\Spwf(\Omega) \to \frakU$ is a closed immersion and $\frakU \to \frakX$ is an open immersion. Without loss of generality, we may assume that $\frakU=\Spwf(A_f^\dagger)$ is a distinguished open subset of $\frakX$. We claim that the map
			\begin{equation*}
				A	\to	A_f^\dagger	\to	\Omega
			\end{equation*}
			is surjective. But this follows from Lemma \ref{l:ff} by passing to the $\frakm$-adic completion.
		\end{proof}
		
		\begin{lemma}\label{l:funct}
			Let $f:\frakX \to \frakY$ be a moprhism in $\FS_R^*$. If $P$ is a rigid point of $\frakX$, then the composition $f \circ P$ factors through a unique rigid point $f(P)$ of $\frakY$.
		\end{lemma}
		\begin{proof}
			We may assume that $\frakX=\Spwf(A)$ and $\frakY=\Spwf(B)$ are affine. Say $P:A \to \Omega$, and let $\Omega'$ the image of $B$ in $\Omega$. Since $\Omega$ is finite over $\Omega'$, $\Omega'$ is an $R$-point, so we may take $f(P):B \to \Omega'$.
		\end{proof}
		
		As a consequence of Lemma \ref{l:funct}, we see that the assignment $\frakX \mapsto \langle \frakX \rangle$ defines a functor $\FS_R^* \to \Set$. The following lemma indicates that this functor is not affected by $\frakm$-adic completion:
		
		\begin{proposition}\label{p:comppts}
			The rigid points functors $\FS_R^* \to \Set$ fit into a commutative diagram
			\begin{equation*}
				\begin{tikzcd}
					\FS_R^\dagger	\arrow[dr]	\arrow[r]	&	\FS_R	\arrow[d]	\\
						&	\Set
				\end{tikzcd}
			\end{equation*}
			where the horizontal arrow is the $\frakm$-adic completion functor.
		\end{proposition}
		\begin{proof}
			Let $\frakX$ be an object of $\FS_R^\dagger$. It is clear that there is an injective map $\langle \frakX \rangle \to \langle \frakX_\infty \rangle$. To see that this map is surjective, let $P:\Spf(\Omega) \to \frakX_\infty$ be a rigid point of $\frakX_\infty$. Choose an affine open subset $\frakU=\Spwf(A)$ of $\frakX$ such that $\frakU_\infty = \Spf(A_\infty)$ contains the image of $P$. Then the composition
			\begin{equation*}
				A	\to	A_\infty	\to	\Omega
			\end{equation*}
			has dense image, and so it is necessarily surjective. It is clear that the associated rigid point of $\frakX$ has $\frakm$-adic completion $P$.
		\end{proof}
		
		Now let $\calX = \frakX^\rig$. We would like to speak of the set of rigid points of $\calX$, but first we must show that this is independent of the choice of model:
		
		\begin{proposition}\label{p:blpoints}
			Let $\frakX' \to \frakX$ be an admissible blow-up in $\FS_R^*$. Then the induced map $\langle \frakX' \rangle \to \langle \frakX \rangle$ is bijective.
		\end{proposition}
		\begin{proof}
			We will construct an inverse to the map $\langle \frakX' \rangle \to \langle \frakX \rangle$. We assume without loss of generality that $\frakX=\Spwf(A)$ is affine. Let $P:A \to \Omega$ be a rigid point of $\frakX$. Choose an ideal $I$ of $A$ such that $\frakX' \to \frakX$ is the admissible blow-up of $\frakX$ at $\tilde{I}$. Since the intergral closure $\overline{\Omega}$ is a discrete valuation ring, the ideal $I \overline{\Omega}$ is invertible. It follows that there is a finite integral extension $\Omega \subseteq \Omega'$ such that $I \Omega'$ is invertible. Note that $\Omega'$ is an $R$-point. By the universal property of admissible blow-ups, the map
			\begin{equation*}
				\Spwf(\Omega')	\to	\frakX
			\end{equation*}
			factors through a unique map $\Spwf(\Omega') \to \frakX'$. Choose an affine open subset $\frakV=\Spwf(B)$ of $\frakX'$ containing the image of $\Spwf(\Omega')$. Replacing $\Omega'$ by the image of $B$ under the map $B \to \Omega'$, we obtain the desired rigid point of $\frakX'$.
		\end{proof}
		
		By the universal property of localization, the rigid points functor $\FS_R^* \to \Set$ factors through a unique functor $\Rig_R^*$. The underlying set of \emph{rigid points} of the analytic space $\calX$ is again denoted by $\langle \calX \rangle$. We may represent such a point by a closed immersion $P:x \to \calX$, where $x$ is the object of $\Rig_R^*$ associated to an $R$-point. By (\cite{Abbes}, 4.4.4), the topos $x_\ad$ is a punctual topos, i.e. it is isomorphic to the category of sets. The induced morphism $P:x_\ad \to \calX_\ad$ is therefore a point of the topos $\calX_\ad$. The rigid-points perspective allows us to view $\calX$ as a family of classical analytic spaces in the following way: Let $P:R \to \Omega$ be a rigid point of $R$, and $\calX_P$ the base change of $\calX$ along $P$. Then $\calX_P$ is an object of $\Rig_\Omega^*$, which by Raynaud \cite{Raynaud} (resp. Langer and Muralidharan \cite{Langer}) is the category of quasi-compact and quasi-separated rigid analytic spaces (resp. dagger spaces) over the residue field $k(P)$.
		
		\begin{theorem}\label{t:conserv}
			A moprhism in $\calX_\ad$ is an isomorphism if and only if it induces an isomorphism on stalks at the rigid points of $\calX_\ad$.
		\end{theorem}
		\begin{proof}
			Let $\frakX$ be a model of $\calX$. For any sheaf $\calH$ in $\calX_\ad$ and any admissible blow-up $\varphi:\frakX_\varphi \to \frakX$, let $\calH_\varphi = (\mu_\varphi)_* \calH$. If $P$ is a rigid point of $\calX_\ad$, write $x(\varphi)$ for the specialization of $P$ in $\frakX_\varphi$, and $x=x(\id)$. There is a natural isomorphism
			\begin{equation*}
				P^{-1}\calH	\cong	\varinjlim_{\varphi \in \Bl_\frakX^\op}	(\calH_\varphi)_{x(\varphi)}.
			\end{equation*}
			Let $\hat{\calH} = \pi_* \calH$ be the sheaf associated to $\calH$ in $\frakX_\tot$. By Theorem \ref{t:shad}, for every morphism $\frakX_\varphi \to \frakX_\psi$ in $\Bl_\frakX$, there is an isomorphism
			\begin{equation*}
				f^* \calH_{\psi} \cong \calH_\varphi.
			\end{equation*}
			In particular, taking $f=\varphi$, we see that $(\calH_\varphi)_{x(\varphi)} \cong (\calH_\id)_x$.
			
			Now let $\alpha:\calF \to \calG$ be a morphism in $\calX_\ad$, and suppose that the map $P^{-1}\alpha$ is an isomorphism for all rigid points $P \in \langle \calX \rangle$. By Proposition \ref{p:shad}, $\alpha$ is an isomorphism if and only if the induced map $\alpha_\varphi:\calF_\varphi \to \calG_\varphi$ is an isomorphism for all $\varphi:\frakX_\varphi \to \frakX$. It suffices to check this isomorphism on the stalks at the Zariski points of $\frakX_\varphi$. By the preceding discussion and our assumption on $\alpha$, we know that this is true for any Zariski point which is the specialization of a rigid point of $\calX$. Thus we need only show that the specialization map is surjective, and by Proposition \ref{p:comppts} it suffices to check this on the level of $\frakm$-adic completions. Without loss of generality, we may assume that $\frakX$ is $\frakm$-torsion free. But then the statement follows from (\cite{Abbes}, 3.3.10).
		\end{proof}
		
		Let $\calF$ be an $\calO_\calX$-module, and $P$ a point of $\calX_\ad$. In accordance with the notation of Proposition \ref{p:wbchange}, we will write $\calF_P = P^* \calF$ for the fiber of $\calF$ at $P$, while retaining the notation $P^{-1} \calF$ for the stalk. For a global section $s \in \Gamma(\calX,\calF)$, we write $s_P$ for the image of $s$ in $\calF_P$. One useful consequence of Theorem \ref{t:conserv} is the following:
		
		\begin{corollary}\label{c:conserv}
			Let $\calF$ be an $\calO_\calX$-module and $s \in \Gamma(\calX,\calF)$. Then $s_P = 0$ if and only if $s_P = 0$ for every rigid point $P \in \langle \calX \rangle$.
		\end{corollary}
		\begin{proof}
			Apply Theorem \ref{t:conserv} to the module $s \calO_\calX$. The result follows by Nakayama's lemma.
		\end{proof}
		
\section{Arithmetic Applications}

	\subsection{\texorpdfstring{$\sigma$}{Sigma}-Modules}\label{s:smod}
	
		We now turn to our main question of of understanding the variation of $p$-adic Galois representations in a suitable family. Let $X=\Spec(A_0)$ be a smooth affine variety over a finite field $k$ of characteristic $p$. We assume henceforth that $R$ has characteristic $0$, and that the residue field $R/\frakm = \F_q$ is contained in $k$. Let $A$ be a smooth $R$-algebra such that $A/\frakm A = A_0$. Since $A$ is smooth over $R$, such an $R$-algebra always exists \cite{Elkik}. A \emph{lifting of Frobenius} for $A$ is an $R$-endomorphism of $A$ reducing to the $q$-Frobenius endomorphism of $A_0$. Such a lifting always exists e.g. if $A$ is a w.c.f.g. algebra over $R$. We assume in this section that we are given a fixed lifting of Frobenius $\sigma$.
		
		\begin{definition}
			A \emph{$\sigma$-module} over $A$ is defined to be a pair $(M,\phi)$, where $M$ is a finite projective $A$-module, and
			\begin{equation*}
				\phi:\sigma^*M \to	M
			\end{equation*}
			is an $R$-linear map. We say that $(M,\phi)$ is a \emph{unit-root} $\sigma$-module if $\phi$ is an isomorphism.
		\end{definition}
		
		A morphism $(M,\phi) \to (M',\phi')$ of $\sigma$-modules over $A$ is simply an $A$-module homomorphism $M \to M'$ making the square
		\begin{equation*}
			\begin{tikzcd}
				\sigma^*M	\arrow[d]	\arrow[r]	&	M	\arrow[d]	\\
				\sigma^*M'	\arrow[r]	&	M'
			\end{tikzcd}
		\end{equation*}
		commute. Let $\Mod(\sigma,A)$ denote the category of $\sigma$-modules over $A$, and $\Mod_0(\sigma,A)$ the full subcategory of unit-root $\sigma$-modules. If $B$ is a second $R$-algebra equipped with a lifting of Frobenius $\tau$, and $f:A \to B$ is an $R$-algebra map satisfying $f \circ \sigma = \tau \circ f$, then we have a base-change functor
		\begin{equation*}
			\Mod(\sigma,A) \to \Mod(\tau,B)
		\end{equation*}
		sending $(M,\phi) \mapsto (M \otimes_A B,\phi \otimes \tau)$. For example, if $(M,\phi)$ is a $\sigma$-module over $A$, then base change along the map $A \to A_n = A/\frakm^{n+1}$ yields a $\sigma$-module $(M_n,\phi_n)$ over $A_n$. If $A_\infty$ denotes the $\frakm$-adic completion of $A$, then $(M,\phi)$ prolongs uniquely to a $\sigma$-module over $A_\infty$.
		
		Let $\bar{x}$ be a fixed geometric point of $X$. Our primary interest in $\sigma$-modules is their connection to $p$-adic representations of the fundamental group $\pi_1(X,\bar{x})$. The following theorem is originally due to Katz, in the special case $R=\Z_q$ \cite{Katz}:
		
		\begin{theorem}\label{t:equiv}
			The category $\Mod_0(\sigma,A_\infty)$ is equivalent to the category of finite-rank $R$-representations of $\pi_1(X,\bar{x})$.
		\end{theorem}
		\begin{proof}
			The proof is identical to Katz' theorem, we provide a sketch here. Suppose that $\pi_1(X,\bar{x}) \to \GL(V)$ is a finite-rank $R$-representation. Let $V_n = V/\frakm^{n+1}V$. Then the composition
			\begin{equation*}
				\pi_1(X,\bar{x})	\to	\GL(V)	\to	\GL(V_n)
			\end{equation*}
			classifies a finite \'{e}tale map $A_0 \to B_n$ which lifts to an \'{e}tale map $A_n \to B_n'$. Let
			\begin{equation*}
				M_n'	=	V_n \otimes_{R/\frakm^{n+1}} B_n',
			\end{equation*}
			which is a free $B_n'$-module of finite rank. Equip $M_n'$ with the diagonal $\GL(V_n)$-action. Then $M_n'$ is a $\GL(V_n)$-equivariant object in the stack of projective $B_n$-modules. By uniqueness in the lifting of $\sigma$ to $B_n'$, it follows that $\phi_n'=1 \otimes \sigma:M_n' \to M_n'$ is also $\GL(V_n)$-equivariant. By descent, the we obtain a pair $(M_n,\phi_n)$, where $M_n$ is a projective $A_n$-module of finite rank and $\phi_n:\sigma^*M_n \to M_n$ is an isomorphism. These pairs are compatible with base change along $A_{n+1} \to A_n$, and we may define the $\sigma$-module $(M,\phi)$ to be their inverse limit.
			
			Conversely, suppose that $(M,\phi)$ is a $\sigma$-module over $A_\infty$. For every \'{e}tale map $A_\infty \to B$, let $M^\phi(B)$ be the $\Z_p$-module consisting of the $\phi \otimes \sigma$-invariants in $M \otimes_{A_\infty} B$. This defines a sheaf of $\Z_p$-modules on the \'{e}tale site of $A_\infty$ which is isomorphic to a constant sheaf $\underline{V}$, where $V$ is a free $R$-module of finite rank. The reduction $\underline{V_n}$ is a constant sheaf of free $R/\frakm^{n+1}$-modules on $A_n$, and therefore corresponds an \'{e}tale torsor for the finite (hence affine) group-scheme $\GL(V_n)$. This torsor is representable by an \'{e}tale map $A_n \to B_n'$, and the reduction $A_0 \to B_n$ gives a map $\pi_1(X,\bar{x}) \to \GL(V_n)$. These maps are compatible and therefore yield the desired representation of $\pi_1(X,\bar{x})$.
		\end{proof}
		
		\begin{remark}\label{r:funct}
			The preceding construction is functorial in the following sense: Let $\frakY=\Spf(B_\infty)$ be a smooth affine formal $R$-scheme with special fiber $Y=\Spec(B_0)$. Suppose that $\tau$ is lifting of Frobenius for $B_\infty$, and $f:\frakY \to \frakX$ is compatible with $\sigma$ and $\tau$. Let $\bar{y}$ be a geometric point of $Y$ and $\bar{x}$ its image in $X$. Then we have a map
			\begin{equation*}
				\pi_1(Y,\bar{y})	\to	\pi_1(X,\bar{x}).
			\end{equation*}
			Given a representation $\rho:\pi_1(X,\bar{x}) \to \GL(V)$ as in Theorem \ref{t:equiv}, we have a corresponding $\sigma$-module $(M,\phi)$ over $A_\infty$. The representation $\rho$ pulls back to a representation of $\pi_1(Y,\bar{y})$, whose corresponding $\tau$-module is the base change of $(M,\phi)$ along $A_\infty \to B_\infty$.
		\end{remark}
		
		Let $(M,\phi)$ be a $\sigma$-module over $A_\infty$. For each closed point $x \in |X|$, we let $k(x)$ denote the residue field of $x$ and $\deg(x)=[k(x):\F_q]$ the degree of $x$ over $\F_q$.  Define
		\begin{equation*}
			R(x)	=	\varprojlim_n	W(k(x)) \otimes R/\frakm^n,
		\end{equation*}
		so that $R(x)$ is an $\frakm$-adically complete $R$-algebra with special fiber $k(x)$. Note that $R(x)$ inherits a lifting of Frobenius $F$ from the canonical lifting on $W(k(x))$. By a result of Monsky \cite{Monsky}, the point $x:A_0 \to k(x)$ lifts uniquely to an $R$-algebra map
		\begin{equation*}
			\hat{x}:A_\infty	\to	R(x)
		\end{equation*}
		which is compatible with the given liftings of Frobenius. We refer to $\hat{x}$ as the \emph{Teichm\"{u}ller point} above $x$. The \emph{fiber} of $(M,\phi)$ is defined to be the $F$-module $(M_x,\phi_x)$ over $W(k(x))$ obtained by base change along $\hat{x}$. The module $M_x$ is free of finite rank. Let $E(x)$ denote the matrix of $\phi_x$ with respect to some basis. The map $\phi_x$ is only $F$-linear, but the iterate $\phi_x^{\deg(x)}$ is $R$-linear, with matrix
		\begin{equation}\label{eq:nmatrix}
			\N_{k(x)/\F_q}	E(x)	=	E(x) E(x)^F \cdots E(x)^{F^{\deg(x)}}
		\end{equation}
		In particular, the characteristic polynomial of $\phi_x^{\deg(x)}$ is $F$-invariant, and therefore lies in $R[s]$.
		
		\begin{definition}
			The $L$-function of the $\sigma$-module $(M,\phi)$ is defined to be
			\begin{equation*}
				L(\phi,s)	=	\prod_{x \in |X|}	\frac{1}{\det\left(	I-\phi_x^{\deg(x)}s^{\deg(x)}	\right)}	\in	R[[s]]
			\end{equation*}
		\end{definition}
		
		With $E(x)$ defined as above, a routine computation shows that we may also write
		\begin{equation*}
			L(\phi,s)	=	\exp\left(	-\sum_{\ell=1}^\infty	\sum_{x \in X(\F_{q^\ell})}	\frac{\tr \N_{\F_{q^\ell}/\F_q} E(x)}{\ell} s^\ell	\right).
		\end{equation*}
		
		Suppose that $(M,\phi)$ is a unit-root $\sigma$-module. Let $x$ be a closed point and $\bar{x} \to x$ a geometric point above $x$. Then $(M,\phi)$ corresponds to a representation $\rho:\pi_1(X,\bar{x}) \to \GL(V)$, and the fiber $(M_x,\phi_x)$ corresponds to the pullback
		\begin{equation*}
			\rho_x:	\pi_1(x,\bar{x})	\to	\pi_1(X,\bar{x})	\to	\GL(V).
		\end{equation*}
		Note that $\rho_x$ sends the canonical generator $F^{\deg(x)}$ to the action of $\Frob_x$ on $V$. It follows that $\rho(\Frob_x) = \phi_x^{\deg(x)}$, and so $L(\phi,s)$ agrees with the $L$-function of the representation $\rho$ over $X$.
		
		\begin{definition}
			Two $\sigma$-modules over $A_\infty$ are \emph{equivalent} if they have the same $L$-function. If $A$ is a w.c.f.g. algebra, an \emph{overconvergent} $\sigma$-module is defined to be any $\sigma$-module over $A_\infty$ which is equivalent to the $\frakm$-adic completion of a $\sigma$-module over $A$.
		\end{definition}
		
		Note that equivalence of $\sigma$-modules may be checked on the fibers at the closed points in $|X|$. In fact, it suffices to check that the matrices (\ref{eq:nmatrix}) have the same characteristic polynomial.
	
	\subsection{Eigenvarieties}\label{s:eig}
	
		Let $K$ be a discrete valuation field of characteristic $0$. In \cite{Monsky}, Monsky develops a spectral theory for certain \emph{nuclear operators} on a $K$-vector space $V$. Our goal in this section is to globalize this theory, replacing $K$ by our rigid analytic weight space $\calW$ and $V$ by an $\calO_\calW$-module $\calV$. Given a rigid point $P \in \langle \calW \rangle$, the fiber $\calV_P$ is a vector space over the residue field $k(P)$. A nuclear operator in our sense induces a nuclear operator $\psi_P:\calV_P \to \calV_P$ for each $P$. For each such operator, there is a well defined \emph{Fredholm determinant} $C(\psi_P,s)\in 1+k(P)[[s]]$, which is a $p$-adically entire power series whose vanishing locus in an algebraic closure $k(P)^s$ of $k(P)$ is precisely the set of non-zero eigenvalues of $\psi_P$ acting on $\calV_P$. In our relative theory, the spectral theory of the operator $\psi$ is described by an \emph{eigenvariety} $\calE(\psi)$, which is a rigid analytic space over $\calW$ defined as the vanishing locus of a similar Fredholm determinant.
		
		Our analogue of the power series $C(\psi_P,s)$ will be an analytic function on the relative multiplicative group $\G_{m,\calW}$ over $\calW$. As this is our first exacmple of a rigid analytic space which is not quasi-compact, we begin by outlining its construction. Choose a set of generators $\frakm = (\pi_1,...,\pi_d)$, and let $\frakW = \Spf(R)$. Recall that the admissible formal blow-up $\frakW_\frakm$ admits a covering by affine formal $R$-schemes $\frakW_i = \Spf(R_i)$ such that $\frakm R_i = \pi_i R_i$. We will refer to the admissible covering $\{\calW_i \to \calW\}_i$ as the \emph{standard covering} of $\calW$ associated to $\pi_1,...,\pi_d$. For each $n > 0$, define the formal $R$-scheme $\frakA_{i,n}	=	\Spf(A_{i,n})$, where
		\begin{equation*}
			A_{i,n}	=	R_i\langle \pi_i^n s, \pi_i^n s^{-1}\rangle.
		\end{equation*}
		For a fixed $n$, the individual $\frakA_{i,n} \to \frakW_i$ glue to give a morphism $\frakA_n \to \frakW_\frakm$ in $\FS_R$. Write $\calA_n \to \calW$ for the image of this morphism in $\Rig_R$. The natural maps $A_{i,n+1} \to A_{i,n}$ define open immersions $\calA_n \to \calA_{n+1}$. We define the \emph{multiplicative group} over $\calW$ to be the ind-object
		\begin{equation}\label{eq:gmpres}
			\G_{m,\calW}:	\calA_1	\to	\calA_2	\to	\cdots.
		\end{equation}
		The multiplicative group may be regarded as group object in the big category $\underline{\Rig}_R$. For example, the inversion map $i:\G_{m,\calW} \to \G_{m,\calW}$ is induced by the maps $\pi_i s \mapsto \pi_i s^{-1}$ on each $A_{i,n}$. The presentation (\ref{eq:gmpres}) allows us to equip $\G_{m,\calW}$ with the structure of a ringed topos. In particular, an analytic function on $\G_{m,\calW}$ is nothing more than a compatible collection of analytic functions on $\calA_n$. If $\calU \to \calW$ is an open immersion, we define $\G_{m,\calU} = \calU \times_{\calW} \G_{m,\calW}$.
		
		\begin{definition}\label{d:cc}
			Let $\calU \to \calW$ be an open immersion, and $\psi:\calV \to \calV$ a linear operator. We say that $\psi$ is \emph{nuclear} over $\calU$ if
			\begin{enumerate}
				\item	For every rigid point $P \in \langle \calU \rangle$, the operator $\psi_P$ is nuclear.
				\item	There is an analytic function $C(\psi|\calU,s)$ on $\G_{m,\calU}$ such that for every rigid point $P \in \langle \calU \rangle$,
					\begin{equation*}
						C(\psi|\calU,s)_P	=	C(\psi_P,s).
					\end{equation*}
			\end{enumerate}
		\end{definition}
		
		By Corollary \ref{c:conserv}, the analytic function $C(\psi|\calU,s)$ is necessarily unique; we will refer to this function as the \emph{Fredholm determinant} of $\psi$ over $\calU$. We will say that $\psi$ is \emph{nuclear} if it is nuclear over $\calW$, and write $C(\psi,s)$ for its Fredholm determinant. The following lemma indicates the property of being nuclear is a local one:
		
		\begin{lemma}\label{l:nucloc}
			Let $\{\calU_i \to \calU\}$ be an admissible covering. Then $\psi$ is nuclear over $\calU$ if and only if $\psi$ is nuclear over $\calU_i$ for each $i$.
		\end{lemma}
		\begin{proof}
			The forward direction is clear, so we will prove the converse. If $P$ is a rigid point of $\calU$, then $P$ is necessarily factors through a rigid point of $\calU_i$ for some $i$. It follows that condition (i) of Definition \ref{d:cc} is satisfied. To show condition (ii), we need only show that the Fredhom determinants $C(\psi|\calU_i,s)$ agree on double intersections. By Corollary \ref{c:conserv}, it suffices to check that their values agree at every $P \in \langle \calU_i \cap \calU_j \rangle$. But for any such $P$, we have
			\begin{equation*}
				C(\psi|\calU_i,s)_P	=	C(\psi_P,s)	=	C(\psi|\calU_j,s)_P.
			\end{equation*}
		\end{proof}
		
		\begin{definition}
			Let $\psi:\calV \to \calV$ be a nuclear operator over $\calU$. The \emph{eigenvariety} of $\psi$ over $\calU$ is the hypersurface $\calE(\psi|\calU)$ in $\A^1_\calW$ cut out by $C(\psi|\calU,s)$.
		\end{definition}
		
		We defer to (\cite{Abbes}, 4.8.29) the construction of a hypersurface as a rigid analytic space over $\calW$. The essential property of $\calE(\psi|\calU)$ is that for every rigid point $P \in \langle \calU \rangle$, the set of rigid points of the fiber $\calE(\psi|\calU)_P$ is the set of $\Gal(k(P)^s/k(P))$-orbits of the non-zero eigenvalues of $\psi_P$. By its construction, the eigenvariety fits into a diagram
		\begin{equation*}
			\begin{tikzcd}
				\calE(\psi|\calU,s)	\arrow[d,"\wt"']	\arrow[r]	&	\G_{m,\calW}	\\
				\calU	&	
			\end{tikzcd}
		\end{equation*}
		We refer to the vertical arrow as the \emph{weight map}. Define $\alpha$ to be the composition
		\begin{equation*}
			\calE(\psi|\calU)	\to	\G_{m,\calU}	\xrightarrow{i} \G_{m,\calU}.
		\end{equation*}
		Let $x$ be a rigid point of $\calE(\psi|\calU)$, and $P=\wt(x)$. Then $\alpha(x)$ may be regarded as a point in the multiplicative group $\G_{m,k(P)}$. The \emph{slope} of $x$ is defined to be the valuation of $\alpha(x)$ in $k(P)$, normalized so that the value group is $\Z$. In \cite{Davis}, the authors have shown a remarkable ``spectral halo'' property describing the relationship between the weight parameter and the slope map near the boundary of their weight space. We now consider some possible analogues of this property in our more general setup.
		
		Define $\frakW_0 = \Spf(\Z_p[[T]])$. Then $\frakW_0$ is an object in $\FS$, and we define $\calW_0 = \frakW_0^\rig$. Note that the set of continuous maps
		\begin{equation*}
			\Z_p[[T]]	\to	R
		\end{equation*}
		is naturally identified with the ideal $\frakm$ in $R$. If $t \in \frakm$, the corresponding map $\frakW \to \frakW_0$ in $\FS^+$ is usually not a morphism in $\FS$, and so does not induce a map $\calW \to \calW_0$ in $\Rig$. Given an open immersion $\calU \to \calW$ in $\Rig_R$, we will say that $t$ is \emph{admissible} in $\calU$ if $\calU \to \calW$ admits a model $\frakU \to \frakW$ such that the composition
		\begin{equation*}
			\frakU \to \frakW	\xrightarrow{t}	\frakW_0
		\end{equation*}
		is a morphism in $\FS$.
		
		\begin{proposition}
			Let $t \in \frakm$. There is a unique maximal open immersion $\calW_t \to \calW$ in $\Rig_R$ such that $s$ is admissible in $\calW_t$.
		\end{proposition}
		\begin{proof}
			Choose a set of generators $\frakm=(\pi_1,...,\pi_d)$, and let $\calW_i = \frakW_i^\rig$ denote the standard coverings associated to $\pi_1,...,\pi_d$. Since $t \in \frakm$, for each $i$ there exists $m_i$ such that $t = \pi_i^{m_i} r_i$, with $r_i \in R_i$. Define
			\begin{equation*}
				S_i = R_i\langle r_i^{-1} \rangle,
			\end{equation*}
			and let $\frakW_{t,i} = \Spf(S_i)$. The maps $\frakW_{t,i} \to \frakW_i$ glue in the evident manner to a map $\frakW_t \to \frakW_\frakm$, and the corresponding map $\calW_t \to \calW$ in $\Rig_R$ is an open immersion. To complete the proof, it suffices to show that any open immersion $\calU \to \calW$ in $\Rig_R$ for which $t$ is admissible in $\calU$ factors through $\calW_t \to \calW$. This problem is local, so we may assume that $\calU = \frakU^\rig$, where $\frakU = \Spf(A)$ and $\frakm A = \pi_i A$ is invertible. Then by the universal property of admissible blowing up, there is a unique morphism $\frakU \to \frakW_\frakm$ whose image lies in $\frakW_i$. Thus we have a corresponding map of $R$-algebras $R_i \to A$. But by assumption, $tA$ is an ideal of definition and therefore $r_i$ must be a unit in $A$. It follows that we get a map $S_i \to A$, and the desired map $\calU \to \calW_t$ is given by the composition
			\begin{equation*}
				\frakU	\to	\frakW_{t,i}	\to	\frakW_t.
			\end{equation*}
		\end{proof}
		
		We will refer to $\calW_t$ as the \emph{admissible locus} of $t$ in $\calW$. By definition, $t$ uniquely defines a map
		\begin{equation*}
			\calW_t	\to	\calW_0.
		\end{equation*}
		If $P$ is a rigid point in $\calW_t$, then the rigid point $t(P)$ corresponds to a map $\Z_q[[T]] \to k(P)$. We write $v_t(P)$ for the valuation of $T$ in $k(P)$. Whenever $\calU \to \calW_0$ is an open immersion, we let $\calU_t = \calW_t \times_{\calW_0} \calU$ denote its pullback to $\calW$.
		
		\begin{definition}
			Let $\psi:\calV \to \calV$ be a nuclear operator, and $\calU \to \calW_0$ an open immersion. We say that $\psi$ is \emph{slope-uniform} over $\calU$ if there exists $t \in \frakm$, $d > 0$, and a rational number $r > 0$ such that:
			\begin{enumerate}
				\item	There is a decomposition
					\begin{equation*}
						\calE(\psi,s) \times_\calW \calU_t	=	\calE_{[0,1)}	\coprod \calE_{[1,2)} \coprod \cdots
					\end{equation*}
					where each $\calE_{[n,n+1)}$ is finite and flat over $\calU_t$ of degree $d$.
				\item	If $x \in \langle \calE_{[n,n+1)} \rangle$ is a point of weight $P$, then
					\begin{equation*}
						v(x)	\in	r v_t(P) \cdot [n,n+1).
					\end{equation*}
			\end{enumerate}
			We say that $\psi$ is \emph{slope-stable} over $\calU$ if $\psi$ is slope-uniform over $\calU$, and there exist rational $\beta_1,...,\beta_d \in [0,1)$ such that
			\begin{enumerate}
				\item	Each $\calE_{[n,n+1)}$ admits a decomposition
					\begin{equation*}
						\calE_{[n,n+1)}	=	\coprod_{j=1}^d	\calE_{[n,n+1),j}.
					\end{equation*}
				\item	If $x \in \langle \calE_{[n,n+1)} \rangle$ is a point of weight $P$, then
					\begin{equation*}
						v(x)	=	r v_t(P)(n+\beta_j).
					\end{equation*}
			\end{enumerate}
		\end{definition}
		
		In brief, the operator $\psi$ is slope-stable over $\calU$ if for every $P \in \langle \calU_t \rangle$, the slopes of the eigenvalues of $\psi_P$ form a fixed finite number of arithmetic progressions, each of whose increment depends only on $t$ and the weight parameter. The weaker property of slope uniformity only requires that these slopes form an approximate union of arithmetic progressions. Either property is only known in a limited number of cases: The main result of \cite{Davis} states that their eigencurve is slope-uniform over $\calW_0\backslash \{ 0 \}$. In \cite{Ren}, the authors prove that their eigenvariety arising from $\Z_p^d$-coverings of $\A^1_k$ is slope-uniform over the same region. It would be interesting to know the precise relationship between these two properties, in particular, does slope uniformity imply slope stability?
		
	\subsection{Dwork Operators}\label{s:dwork}
	
		We now introduce our main examples of nuclear operators. Let $A$ be a w.c.f.g. algebra over $R$, and $\sigma$ a lifting of Frobenius for $A$. If $M$ is a finite $A$-module, we define a \emph{Dwork operator} on $M$ to be an $R$-linear map
		\begin{equation*}
			\Theta:\sigma_*M	\to	M.
		\end{equation*}
		Alternatively, we may regard a Dwork operator as an $R$-endomorphism $\Theta:M \to M$ satisfying the relation $\Theta(\sigma(a)m) = a \Theta(m)$ for all $a \in A$ and $m \in M$. Let $\frakX = \Spwf(A)$, $\calX=\frakX^\rig$, and $\calM$ the coherent $\calO_\calX$-module associated to $M$. If $\pi:\calX \to \calW$ denotes the structure morphism, then we define the $\calO_\calW$-module $\calV = \pi_* \calM$. Our aim in this section is to prove the following:
		
		\begin{theorem}\label{t:nuc}
			The action of $\Theta$ on $\calV$ is nuclear. Moreover, the Fredholm determinant $C(\Theta,s)$ lies in $R[[s]]$.
		\end{theorem}
		
		Let $P:R \to \Omega$ be a weight. The fiber $\calX_P$ is the object of $\Rig_\Omega^\dagger$ associated to the weak formal scheme $\frakX_P = \Spwf(A_P)$, where $A_P = A \otimes_R \Omega$. The action of $\sigma$ extends to $A_P$ in the obvious way. Writing $M_P = M \otimes_A A_P$, the induced map $\Theta_P$ is an $\Omega$-linear Dwork operator on $M_P$. Monsky \cite{Monsky} has proven that $\Omega_P$ induces a nuclear operator on the $k(P)$-vector space
		\begin{equation*}
			\calV_P	=	P^* \calV	=	M_P	\otimes_\Omega	k(P).
		\end{equation*}
		Thus $\Theta$ satisfies condition (i) of Definition \ref{d:cc}. The main difficulty will be the construction of the Fredholm determinant $C(\Theta,s)$. We begin with a reduction process: Choose a presentation
		\begin{equation*}
			R[X_1,...,X_n]^\dagger \twoheadrightarrow A.
		\end{equation*}
		We may regard $M$ as a module over the ring $B=R[X_1,...,X_n]^\dagger$. Choose a resolution
		\begin{equation*}
			F_\bullet	\to	M
		\end{equation*}
		of $M$ by free $B$-modules. The functor $\sigma_*$ is exact, and sends free modules to free modules. It follows that $\sigma_* F_\bullet \to \sigma_* M$ is a free resolution, and therefore we have induced Dwork operators $\Theta_k:\sigma_* F_k \to F_k$ for each $k$.
		
		\begin{lemma}
			If $\Theta_k$ is nuclear for each $k$, then $\Theta$ is nuclear and
			\begin{equation}\label{eq:section}
				C(\Theta,s)	=	\prod_k	C(\Theta_k,s)^{(-1)^{k-1}}.
			\end{equation}
		\end{lemma}
		\begin{proof}
			By Corollary \ref{c:conserv}, it suffices to check that for every $P \in \langle R \rangle$, the image of (\ref{eq:section}) in $k(P)$ is equal to $C(\Theta_P,s)$. But this follows from (\cite{Monsky}, 1.4).
		\end{proof}
		
		In light of the Lemma, we may assume that $B=A$ and that $M$ is a free $A$-module of finite rank, say $M = \bigoplus_j A$. Let $q_j:A \to M$ denote the $j$th canonical injection. Choose a set of generators $\frakm = (\pi_1,...,\pi_d)$, and let $\calW_i = \Spf(R_i)^\rig$ be the standard covering associated to $\pi_1,...,\pi_d$. By Lemma \ref{l:nucloc}, it suffices to prove that $\Theta$ is nuclear over $\calW_i$ for each $i$. Let $M_i = M \otimes_A A_i$, and $\Theta_i$ the $R_i$-linear Dwork operator on $M_i$ induced by $\Theta$. Then $\Theta$ and $\Theta_i$ induce the same operator on the restriction $\calV|_{\calW_i}$. We equip the field $Q_i = Q(R_i)$ with the $\pi_i$-adic valuation. By Monsky, $\Theta_i$ induces a nuclear operator on the vector space
		\begin{equation*}
			\calV_i	=	M_i	\otimes_{R_i}	Q_i.
		\end{equation*}
		At this point, we may define $C(\Theta_i|\calW_i,s) \in Q_i[[s]]$ to be the Fredholm determinant of this operator, but it is not at all obvious that this defines an analytic function on $\G_{m,\calW}$. To proceed, consider the family of subspaces indexed by $c>0$:
		\begin{equation*}
			M_i^{(c)}	=	\left\{	\sum_j	r_{u,j}\pi_i^{\lfloor |u|/c \rfloor} q_j(X^u):r_{u,j} \in R_i	\right\}.
		\end{equation*}
		Then $M_i$ is the union of all $M_i^{(c)}$, and each $M_i^{(c)}$ is a free $R_i$-module, with basis consisting of $e_{i,j,u}^{(c)} = \pi_i^{\lfloor |u|/c \rfloor} q_j(X^u)$. Let $\langle \cdot,\cdot \rangle$ denote the bilinear form on $M_i^{(c)}$ for which this basis is orthonormal. Define $\calV_i^{(c)} = M_i^{(c)} \otimes_{R_i} Q_i$.
		
		\begin{lemma}\label{l:monsky}
			There exists $c_0 > 0$ such that for all $c \geq c_0$ $\calV_i^{(c)}$ is stable under the action of $\Theta_i$, and the restriction of $\Theta_i$ to $\calV_i^{(c)}$ is a compact operator.
		\end{lemma}
		\begin{proof}
			This is (\cite{Monsky}, 2.4).
		\end{proof}
		
		To complete the proof of Theorem \ref{t:nuc}, consider the basis of $\calV_i$ defined by
		\begin{equation*}
			e_{i,j,u} = q_j(X^u).
		\end{equation*}
		Since $M_i$ is stable under $\Theta_i$, the matrix $C$ of $\Theta_i$ with respect to this basis has coefficients in $R_i$. For each $c > c_0$, let $C^{(c)}$ denote the matrix of $\Theta_i$ with respect to the basis $\{e_{i,j,u}^{(c)}\}$. The element $C(\Theta_i|\calW_i,s)$ is defined as the limit of Fredholm determinants of certain finite submatrices of $C^{(c)}$. Since the corresponding finite submatrices of $C$ are similar, it follows that these power series are equal and therefore are elements of $R_i[[s]]$. As $R_i$ is complete, we see that $C(\Theta_i|\calW_i,s) \in R_i[[s]]$. Now if we expand
		\begin{equation*}
			C(\Theta_i|\calW_i,s)	=	r_{i,0}	+ r_{i,1}s	+\cdots,
		\end{equation*}
		then by (\cite{Serre}, 5.7) we have that
		\begin{equation*}
			\lim_{j \to \infty}	\frac{v_{\pi_i}(r_{i,j})}{j}	=	\infty.
		\end{equation*}
		In other words, the series $C(\Theta_i|\calW_i,s)$ is entire with respect to the $\pi_i$-adic valuation. In particular, it defines an analytic function on $\G_{m,\calW_i}$. This completes the proof that $\Theta$ is nuclear. As the power series $C(\Theta_i|\calW_i,s)$ agree on double intersections, we see that for each $j$, the $r_{i,j}$ glue to give a global section $r_j$ of the admissible formal blow-up $\Spf(R)_\frakm$. But from the definition of admissible formal blow-up, we see that the ring of global sections is precisely $R$. Thus $C(\Theta|\calW,s) \in R[[s]]$ as desired.
		
		\begin{definition}
			For each $d > 0$, the \emph{trace} of the operator $\Theta^d$ is the element of $R$ defined by the relation
			\begin{equation*}
				C(\Theta,s)	=	\exp\left(	-\sum_{d=1}^\infty	\Tr(\Theta^d) \frac{s^d}{d}	\right).
			\end{equation*}
		\end{definition}
		
	\subsection{The Trace Formula}\label{s:trace}
	
		Again, let $X=\Spec(A_0)$ be a smooth affine variety over $k$, and $A$ a smooth w.c.f.g. algebra over $R$ lifting $A_0$. Fix a lifting of Frobenius $\sigma:A \to A$, and let $(M,\phi)$ be a $\sigma$-module over $A$. We are now ready to prove our main theorem regarding the meromorphic continuation of the $L$-function $L(\phi,s)$. Following Monsky and Washnitzer \cite{MW}, we will construct Dwork operators on the de Rahm complex of the dual module $M^\vee$, and prove that $L(\phi,s)$ can be written as an alternating product of their Fredholm determinants. At each rigid point $P \in \langle \calW \rangle$, our result reduces to Monsky's trace formula \cite{Monsky}.
		
		Write $\Omega^\bullet A$ for the complex of \emph{continuous} differentials of $A$. By (\cite{MW}, 4.6), each $\Omega^i A$ is a finite projective $A$-module. For brevity, let $B=\sigma(A)$. By Nakayama's lemma and smoothness of $A$, we see that $A$ is a finite projective $B$-module of rank $q^{\dim(X)}$. In particular, the localization $A_\frakm$ is a finite free $B_\frakm$-module, and there is a well defined $B_\frakm$-linear trace map
		\begin{equation*}
			\Tr:A_\frakm	\to	B_\frakm.
		\end{equation*}
		Our first goal is to show that the restriction of $\Tr$ to $A$ prolongs to a map of complexes $\Tr_\bullet:\Omega^\bullet A \to \Omega^\bullet B$. First, we need:
		
		\begin{lemma}
			The canonical map $\Omega^\bullet B \to \Omega^\bullet A$ induces an isomorphism
			\begin{equation}\label{eq:isom}
				\Omega^\bullet	B	\otimes_B	A_\frakm	\xrightarrow{\sim}	\Omega^\bullet	A	\otimes_A	A_\frakm.
			\end{equation}
		\end{lemma}
		\begin{proof}
			Both modules are free $A_\frakm$-modules of the same (finite) rank. By (\cite{MW}, 8.1), the induced map
			\begin{equation*}
				\Omega^\bullet	B	\otimes_B	Q(A)	\to	\Omega^\bullet	A	\otimes_A	Q(A).
			\end{equation*}
			is an isomorphism. In particular, (\ref{eq:isom}) is necessarily injective. To see that this map is surjective, note that its reduction mod $\frakm$ is an isomorphism. The result follows by Nakayama's lemma.
		\end{proof}
		
		Let $\alpha_\bullet$ denote the inverse of the isomorphism (\ref{eq:isom}). Then $\Tr_\bullet$ is defined to be the composition
		\begin{equation}\label{eq:trdef}
			\Omega^\bullet A \otimes_A A_\frakm	\xrightarrow{\alpha_\bullet}	\Omega^\bullet	B	\otimes_B	A_\frakm	\xrightarrow{1 \otimes \Tr}	\Omega^\bullet B \otimes_B B_\frakm.
		\end{equation}
		
		\begin{theorem}
			The map $\Tr_\bullet$ maps $\Omega^\bullet A$ into $\Omega^\bullet B$.
		\end{theorem}
		\begin{proof}
			Let $P:R \to \Omega$ be a rigid point of $R$. Write $A_P = A \otimes_R \Omega$ and similarly for $B$. By base change along $R \to \Omega \to k(\Omega)$, we obtain a map of complexes
			\begin{equation}\label{eq:trcomp}
				\Tr_\bullet:	\Omega^\bullet A_P \otimes_{A_P} Q(A_P)	\to	\Omega^\bullet B_P	\otimes_{B_P}	Q(B_P).
			\end{equation}
			We will show first that this restricts to a map $\Omega^\bullet A_P \to \Omega^\bullet B_P$. Let $\overline{\Omega}$ be the integral closure of $\Omega$ in $Q(\Omega)$, so that $\overline{\Omega}$ is a discrete valuation ring. Write $\overline{A}_P	=	A_P	\otimes_\Omega	\overline{\Omega}$ and similarly for $B_P$. Then $\overline{A}_P$ and $\overline{B}_P$ are smooth and integrally closed w.c.f.g. algebras over $\overline{\Omega}$ (\cite{MW}, 6.3). Regarding (\ref{eq:trcomp}) as a map
			\begin{equation*}
				\Omega^\bullet \overline{A}_P \otimes_{\overline{A}_P} Q(A_P)	\to	\Omega^\bullet \overline{B}_P	\otimes_{\overline{B}_P}	Q(B_P),
			\end{equation*}
			we see that this is precisely the trace map of Monsky and Washnitzer, which maps $\Omega^\bullet \overline{A}_P \to \Omega^\bullet \overline{B}_P$ (\cite{MW}, 8.1-8.3). Thus the image under $\Tr_i$ of $\Omega^i A_P$ lies in $\Omega^i \overline{B}_P$. But by (\ref{eq:trdef}), this image also lies in $\Omega^i B_P \otimes_{B_P} B_P'$. where $B_P'$ is the localization of $B_P$ at the maximal ideal of $\Omega$. But clearly $\overline{B}_P \cap B_P' = B$ and so $\Tr_\bullet: \Omega^\bullet A_P \to \Omega^\bullet B_P$ as desired.
			
			Now $\Omega^\bullet B \otimes_B B_\frakm$ is a finite free $B_\frakm$-module, and moreover admits a basis of elements of $B$. It suffices then to prove the following statement: If $b \in B_\frakm$ is such that the image $b_P$ of $b$ in $B_P'$ lies in $B_P$ for all rigid points $P:R \to \Omega$ of $R$, then $b \in B$. For such a $b$, we want to show that inclusion $B \to B[b]$ is an equality, or equivalently that $B \to B[b]^\dagger$ is an equality, since $B$ is weakly complete. Let $f:\frakX \to \frakY$ be the corresponding map in $\FS^\dagger$. Then $f^\rig$ is an open immersion, so in particular there exists a diagram
			\begin{equation*}
				\begin{tikzcd}
					\frakX' \arrow[d]	\arrow[r,"f'"]	&	\frakY'	\arrow[d]	\\
					\frakX	\arrow[r,"f"]	&	\frakY
				\end{tikzcd}
			\end{equation*}
			where $f'$ is an open immersion and the vertical arrows are admissible weak blow-ups. We claim that $f'$ is an isomorphism, from which it will follow that $f^\rig$ is an isomorphism. Let $y$ be a rigid point of $\frakY^\rig$, and $P = \wt(y)$. Then by the assumption on $b$, $f$ induces an isomorphism on the fibers $\frakX_P^\rig	\to	\frakY_P^\rig$. It follows that $f'$ is surjective on rigid points, and passing to specializations that it is surjective on Zariski points, proving the claim. Consequently there is a diagram $\frakX \xleftarrow{\varphi} \frakZ \xrightarrow{\psi} \frakY$, where both arrows are admissible blow-ups. But $\frakX$ and $\frakY$ are affine and $\frakm$-torsion free, thus
			\begin{equation*}
				B	=	\Gamma(\frakZ,\calO_\frakZ)	=	B[b]^\dagger,
			\end{equation*}
			completing the proof.
		\end{proof}
		
		\begin{definition}
			The \emph{canonical Dwork operators} on $\Omega^\bullet A$ are defined to be the composition
			\begin{equation*}
				\theta_\bullet	=	\sigma_\bullet^{-1}	\circ	\Tr_\bullet:\Omega^\bullet A	\to	\Omega^\bullet A.
			\end{equation*}
		\end{definition}
		
		For a $\sigma$-module $(M,\phi)$ over $A$, let $\Omega^i M = \omega^i A \otimes_A M$. The exterior product defines a perfect pairing
		\begin{equation*}
			\Omega^i A	\times	\Omega^{n-i} A	\to	\Omega^n A
		\end{equation*}
		via which we identify the dual $(\Omega^i A)^\vee = \Hom_A(\Omega^i A,A)$ with $\Omega^{n-i} A$. This gives us an identification
		\begin{equation*}
			\Omega^{n-i}M^\vee	=	\Hom_A(\Omega^i M,\Omega^i A).
		\end{equation*}
		We define a Dwork operator $\theta_{n-i}(\phi)$ on $\Omega^{n-i} M^\vee$ by sending $f:\Omega^i M \to \Omega^n A$ to the $A$-linear map
		\begin{equation*}
			\Omega^i M	\xrightarrow{\sigma \otimes \phi}	\Omega^i M	\xrightarrow{f}	\Omega^n A	\xrightarrow{\theta_n}	\Omega^n A.
		\end{equation*}
		
		\begin{theorem}
			Let $(M,\phi)$ be as above. Then
			\begin{equation*}
				L(\phi,s)	=	\prod_{i=0}^n	C(\theta_{n-i}(\phi),s)^{(-1)^{i-1}}.
			\end{equation*}
		\end{theorem}
		\begin{proof}
			It suffices to check the equality at the rigid points of $R$. Let $P:R \to \Omega$ be a rigid point of $R$ and $M_P$ the base change of $M$ along $P$. The map $\phi$ induces a $\sigma$-module structure on $M_P \otimes_\Omega \overline{\Omega}$, and the $L$-function of this $\sigma$-module agrees with that of $(M,\phi)$. On the other hand, the Dwork operators $\theta_\bullet(\phi)$ extend to Dwork operators on $\Omega_\bullet \overline{A}_P$ and the traces of these operators depend only on their action on $\Omega_\bullet A_P \otimes_\Omega k(P)$. Consequently we may assume that $\Omega$ is a discrete valuation ring, but then the statement is Monsky's trace formula as given in \cite{Wan}.
		\end{proof}
		
	\subsection{\texorpdfstring{$T$}{T}-adic Exponential Sums}\label{s:hodge}
	
		In this final section we connect our trace formula to the theory of $T$-adic exponential sums, as introduced in \cite{Liu}. We define these sums in the following much more general setting: let $X=\Spec(A_0)$ be a smooth affine variety over $k$. Let $A$ be a w.c.f.g. algebra over $\Z_p$ with special fiber $A_0$, and let $\sigma:A \to A$ be a lifting of the absolute Frobenius. For any element $f \in A_\infty$, we define the \emph{degree-$d$ $T$-adic exponential sum}
		\begin{equation*}
			S_f(T,d)	=	\sum_{x \in X(\F_q^d)}	(1+T)^{\Tr_{\F_{q^d}/\F_p} f(\hat{x})}	\in	\Z_p[[T]]
		\end{equation*}
		Here, we identify the Galois group of $\F_{q^d}/\F_p$ with the automorphism group of $\Z_{q^d}/\Z_p$. To study these sums, we define the generating function
		\begin{equation*}
			L_f(T,s)	=	\exp\left(	-\sum_{d=1}^\infty	S_f(\chi,d)	\frac{s^d}{d}	\right).
		\end{equation*}
		Let $R=\Z_p[[T]]$. Given a rigid point $P:R \to \Omega$ of $\calW$, the specialization $S_f(P(T),d)$ is a classical ($p$-adic) exponential sum with values in $\Omega$. We regard the sum $S_f(T,d)$ as a family of classical exponential sums, parameterized by the rigid points of $\calW$. Our goal is to understand the $p$-adic variation of the $L$-functions $L_f(P(T),s)$ as $P$ varies.
		
		We begin by discussing how $T$-adic exponential sums arise naturally from $\Z_p$-towers. Fix a geometric point $\bar{x}$ of $X$. We define a \emph{$\Z_p$-tower} to be a continuous homomorphism
		\begin{equation*}
			\pi_1(X,\bar{x})	\to	\Z_p.
		\end{equation*}
		Write $A_n = A/\frakm^{n+1} A$ for all $n \geq 0$. By the lifting property of \'{e}tale morphisms, the small \'{e}tale site $A_\et$ of $A_\infty$ is equivalent to that of $A_n$ for all $n$. Consider the \'{e}tale sheaf of $\Z_p$-modules $\tilde{\G}_a$ on $A_\et$ represented by the formal group $\G_a=\Spf(\Z_p\langle t \rangle)$. We define a homomorphism $\wp:\tilde{\G}_a \to \tilde{\G}_a$ \'{e}tale-locally via
		\begin{align*}
			\G_a(B)	&\to	\G_a(B)	\\
			b	&\mapsto	\sigma(b)-b
		\end{align*}
		for any \'{e}tale $A_\infty \to B$.
		
		\begin{lemma}\label{l:exact}
			There is an exact sequence of \'{e}tale sheaves of $\Z_p$-modules
			\begin{equation}\label{eq:exact}
				0	\to	\underline{\Z_p}	\to	\tilde{\G}_a	\xrightarrow{\wp}	\tilde{\G}_a	\to	0,
			\end{equation}
			where $\underline{\Z_p}$ denotes the constant \'{e}tale sheaf associated to $\Z_p$.
		\end{lemma}
		\begin{proof}
			It suffices to prove that the statement mod $\frakm$, i.e. that the sequence of \'{e}tale sheaves of $\F_p$-modules
			\begin{equation*}
				0	\to	\underline{\F_p}	\to	\tilde{\G}_a	\xrightarrow{\wp}	\tilde{\G}_a	\to	0
			\end{equation*}
			on the small \'{e}tale site of $X$ is exact. Let $A_0 \to B$ be an \'{e}tale map. Then clearly $b^p-b = 0$ if and only if $b \in \F_p$, so that $\ker(\wp) = \underline{\F_p}$, as desired. Now for any $b \in B$, we may consider the Artin-Schreier covering
			\begin{equation*}
				B_b	=	B[X]/(X^p-X-b),
			\end{equation*}
			which is \'{e}tale over $B$ precisely when $b$ is not of the form $x^p-x$ for some $x \in B$. It follows that $\wp$ is \'{e}tale-locally surjective, completing the proof.
		\end{proof}
		
		\begin{proposition}\label{p:zptow}
			There is an isomorphism of $\Z_p$-modules
			\begin{align*}
				A_\infty/\wp A_\infty	&\xrightarrow{\sim}	\Hom(\pi_1(X,\bar{x}),\Z_p).	\\
				a	&\mapsto	\alpha_a.
			\end{align*}
		\end{proposition}
		\begin{proof}
			Taking \'{e}tale cohomology of the exact sequence (\ref{eq:exact}), we have a long exact sequence
			\begin{equation*}
				0	\to	\Z_p	\to	A_\infty	\xrightarrow{\wp}	A_\infty	\to	H^1(A_\et,\underline{\Z_p})	\to	H^1(A_\et,\G_a)	\to	\cdots.
			\end{equation*}
			Now $H^1(A_\et,\underline{\Z_p}) = \Hom(\pi_1(X,\bar{x}),\Z_p)$, and
			\begin{equation*}
				H^1(A_\et,\G_a)	=	H^1(X,\calO_X)	=	0,
			\end{equation*}
			thus giving the desired isomorphism.
		\end{proof}
		
		\begin{remark}\label{r:zptow}
			The same argument yields an isomorphism
			\begin{equation*}
				A_n/\wp A_n	\xrightarrow{\sim}	\Hom(\pi_1(X,\bar{x}),\Z/p^n\Z).
			\end{equation*}
			We can describe this map explicitly: given $a \in A_n$, Lemma \ref{l:exact} shows that there exists an \'{e}tale map $A_n \to B_n$ with Galois group $\Z/p^n\Z$ such that $\wp(b) = a$ admits a solution in $B_n$. The desired map then sends $g \mapsto (gb-b)$.
		\end{remark}
		
		Let $\alpha_f:\pi_1(X,\bar{x}) \to \Z_p$ denote the $\Z_p$-tower associated to our chosen $f \in A_\infty$. Every continuous character $\chi:\Z_p \to \Omega^\times$ taking values in a $\Z_p$-point pulls back to a character $\chi_f$ of $\pi_1(X,\bar{x})$. We can interpolate all such characters by a single transcendental character
		\begin{align*}
			\lambda:\Z_p	&\to	R^\times	\\
			a	&\mapsto	(1+T)^a.
		\end{align*}
		Then $\lambda$ has the following universal property: for every $\chi:\Z_p \to \Omega^\times$ as above, there is a unique map of $\Z_p$-algebras $P_\chi:R \to \Omega$ making the diagram
		\begin{equation*}
			\begin{tikzcd}
				\Z_p	\arrow[dr,"\chi"']	\arrow[r,"\lambda"]	&	R^\times	\arrow[d,"P_\chi"]	\\
					&	\Omega^\times
			\end{tikzcd}
		\end{equation*}
		commute: explicitly, $P_\chi$ maps $T \mapsto \chi(1)-1$. The assignment $\chi \mapsto P_\chi$ gives a bijection
		\begin{equation*}
			\Hom(\Z_p,\Omega^\times)	\xrightarrow{\sim}	\Hom_{\Z_p}(R,\Omega).
		\end{equation*}
		In other words, the set of rigid points of $\calW$ is precisely the set of continuous $p$-adic characters of $\Z_p$.
		
		\begin{lemma}
			Let $x$ be a closed point of $X$. Then
			\begin{equation*}
				\alpha_f(\Frob_x)	=	\Tr_{k(x)/\F_p}	f(\hat{x}).
			\end{equation*}
		\end{lemma}
		\begin{proof}
			Let $d=[k(x):\F_q]$. Fix a separable closure $k^s/k(x)$, and let $F$ be the absolute Frobenius automorphism of $k^s$. Then $G=\Gal(k^s/k(x))$ is generated by $F^d$. By Proposition \ref{p:zptow}, the element $f(\hat{x}) \in W(k(x))$ classifies a $\Z_p$-extension $k'/k(x)$. Then explicitly,
			\begin{equation*}
				\alpha_f(\Frob_x)	=	\alpha_{f(\hat{x})}(F^d).
			\end{equation*}
			For each $n$, let $k_n$ be the unique intermediate extension of degree $p^n$. Write $f_n$ for the image of $f$ in $A_n$. Then by Remark \ref{r:zptow}, there exists $b_n \in B_n$ such that $\wp(b_n) = f_n(\hat{x})$. The image of $g=F^d$ in $\Z/p^n\Z$ is then
			\begin{equation*}
				F^d b_n - b_n	=	\sum_{j=0}^{d-1} F^j(Fb-b)	=	\Tr_{k(x)/\F_p} f_n(\hat{x}).
			\end{equation*}
			Taking inverse limits, we get the desired result.
		\end{proof}
		
		As a consequence of the lemma, we see that $L_f(T,s)$ is equal to the $L$-function of the universal character $\lambda_f$. Let $A_\infty(R)$ denote the completed base change of $A_\infty$ along $\Z_p \to R$. By Theorem \ref{t:equiv}, $\lambda_f$ corresponds to a rank-$1$ $\sigma$-module $(M_f,\phi_f)$ over $A_\infty(R)$. Our next goal is to determine the structure of this $\sigma$-module. We will construct an explicit map of $\Z_p$-modules
		\begin{align*}
			A_\infty	&\to	1+TA_\infty[[T]]	\\
			a	&\mapsto	E_a
		\end{align*}
		such that $(M_f,\phi_f)$ is equivalent to the free rank-$1$ $\sigma$-module $(A_\infty,E_f \circ \sigma)$. Before giving this construction, we require an easy Witt vector lemma:
		
		\begin{lemma}
			There is a unique map of $\Z_p$-modules $\Delta_\sigma:A_\infty \to W(A_\infty)$ compatible with $\sigma$ and the canonical Frobenius endomorphism of $W(A_\infty)$.
		\end{lemma}
		\begin{proof}
			This follows from (\cite{Haze}, 17.6.8), taking $\phi_{p^i}(a)=\sigma^i(a)$ and $\phi_n(a)=0$ when $n$ is not a power of $p$.
		\end{proof}
		
		We identify the underlying additive group of $W(A_\infty)$ with the $\Z_p$-module $\Lambda(A_\infty) = 1+tA_\infty[[t]]$ in the standard way. Using (\cite{Haze}, 17.6.1), is not difficult to work out the map $\Delta_\sigma:A_\infty \to \Lambda(A_\infty)$ explicitly. Namely for $a \in A_\infty$, we have
		\begin{equation*}
			\Delta_\sigma(a)	=	\exp\left(	\sum_{i=0}^\infty	\frac{\sigma^i(a) t^{p^i}}{p^i}	\right).
		\end{equation*}
		Recall that the \emph{Artin-Hasse exponential series} is the power series
		\begin{equation*}
			E(t)	=	\exp\left(	\sum_{i=0}^\infty	\frac{t^{p^i}}{p^i}	\right)	 \in 1+t+t^2\Z_p[[t]].
		\end{equation*}
		For any $r \in \frakm$, there is a well defined element $E(r) \in R$. By a standard argument, this series defines a bijection
		\begin{equation*}
			\frakm^i	\to	1+\frakm^i.
		\end{equation*}
		In particular, there is a unique $\pi \in \frakm$ such that $E(\pi) = 1+T$. The composition
		\begin{equation*}
			A_\infty	\xrightarrow{\Delta_\sigma}	\Lambda(A_\infty)	\xrightarrow{t \mapsto \pi} 1+\pi A_\infty[[\pi]]	\subseteq A_\infty(R)
		\end{equation*}
		is a $\Z_p$-module homomoprhism which we denote by $a \mapsto E_a$.
		
		\begin{proposition}\label{p:split}
			(Dwork's Splitting Lemma). $(M_f,\phi_f)$ is equivalent to the free rank-$1$ $\sigma$-module $(A_\infty, E_f \circ \sigma)$.
		\end{proposition}
		\begin{proof}
			We need only check that for every closed point $x \in |X|$, the fibers of these $\sigma$-modules at $X$ are equivalent. This amounts to showing the equality
			\begin{equation*}
				\lambda_f(\Frob_x)	=	N_{k(x)/\F_p}	E_f(\hat{x})	=	N_{k(x)/\F_p} E_{f(\hat{x})}.
			\end{equation*}
			Note that we have an expansion
			\begin{equation*}
				f(\hat{x})	=	[c_0]+p[c_1]+\cdots	\in W(k(x)).
			\end{equation*}
			By $\Z_p$-linearity, it suffices to prove the statement when $f(\hat{x}) = [c]$ is the Teichm\"{u}ller lift of an element of $k(x)$. But then
			\begin{equation*}
				\lambda_f(\Tr_{k(x)/\F_p} [c])	=	E(\pi)^{\Tr_{k(x)/\F_p} f(\hat{x})}	=	N_{k(x)/\F_p} E(\pi[c])	=		N_{k(x)/\F_p} E_{[c]}.
			\end{equation*}
		\end{proof}
		
		Now let $A(R)$ be the weak base change of $A$ to a w.c.f.g. algebra over $R$. It follows from the above that the $\sigma$-module $(M_f,\phi_f)$ is overconvergent if and only if the representative $f$ can be chosen such that $E_f \in A(R)$. In this case we can apply Theorem \ref{t:merom} to get an explicit formula for the exponential sums $S_f(T,d)$:
		
		\begin{theorem}
			Suppose that $f \in A_\infty$ is such that $E_f \in A(R)$. Define the Dwork operator $\psi_i = \theta_i \circ E_f$ on $\Omega^i A$. Then
			\begin{equation*}
				S_f(T,d)	=	\sum_{i=0}^n	(-1)^i	\Tr(\psi_i^d).
			\end{equation*}
		\end{theorem}
		
	\bibliographystyle{plainnat}
	\bibliography{WeakBib}
\end{document}